\documentclass[reqno]{amsart}
\usepackage{graphicx}
\usepackage{bm}
\usepackage{hyperref}
\usepackage{amssymb}
\usepackage{amsmath,amsthm, latexsym, wasysym}
\usepackage{enumerate}
\usepackage{amsfonts}
\usepackage{todonotes}
\usepackage{enumitem}
\usepackage{accents}
\usepackage{graphicx}
\newtheorem{theorem}{Theorem}[section]
\newtheorem{lemma}[theorem]{Lemma}

\newtheorem{proposition}[theorem]{Proposition}
\newtheorem{remark}[theorem]{Remark}
\theoremstyle{definition}

\newtheorem{example}[theorem]{Example}

\numberwithin{equation}{section}
\usepackage[T1]{fontenc}
\usepackage{mathrsfs}
\usepackage{setspace}
\usepackage{graphicx,subfigure}
\usepackage{xcolor}
\usepackage{epstopdf}
\usepackage[utf8]{inputenc}
\usepackage{multicol}
\usepackage{amsmath}
\usepackage{amsthm}
\usepackage{upgreek}
\usepackage{amsfonts}
\usepackage{contour}
\usepackage{float}
\usepackage{breqn}
\usepackage{textcomp}
\usepackage{gensymb}
\usepackage[left=.8in,right=.8in,top=.8 in,bottom=.7 in]{geometry}
\usepackage{hyperref}
\usepackage{lipsum}
\usepackage{apptools}
\newcommand{\remove}[1]{}

\theoremstyle{remark}
\newcommand{\uhat}{\underaccent{\check}}

\newcommand{\ds} {\displaystyle}
\newcommand{\e}{\epsilon}

\newcommand{\al} {\alpha}

\newcommand{\de} {\delta}

\newcommand{\Om} {\Omega}
\newcommand{\ra} {\rightarrow}

\newcommand{\De} {\Delta}
\newcommand{\la} {\lambda}
\newcommand{\La} {\Lambda}
\newcommand{\noi} {\noindent}
\newcommand{\na} {\nabla}

\newcommand{\mb} {\mathbb}
\newcommand{\mc} {\mathcal}

\begin{document}
	
	\title[Quasilinear Schr\"odinger-Choquard equation]{Quasilinear Schr\"odinger equations with Stein-Weiss type  convolution and critical exponential nonlinearity in $\mathbb R^N$}
	\author[Reshmi Biswas, Sarika Goyal and  K. Sreenadh]
	{Reshmi Biswas, Sarika Goyal and K. Sreenadh}
	\address{Reshmi Biswas \newline
		Department of Mathematics, IIT Delhi, Hauz Khas, New Delhi 110016, India}
	\email{reshmi15.biswas@gmail.com}

	\address{Sarika Goyal \newline
		Department of Mathematics, Netaji Subhas University of Technology,
		Dwarka Sector-3, Dwarka, Delhi, 110078, India}
	\email{sarika1.iitd@gmail.com }
	
	\address{K. Sreenadh \newline
		Department of Mathematics, IIT Delhi, Hauz Khas, New Delhi 110016, India}
	\email{sreenadh@maths.iitd.ac.in}
	
	
	\begin{abstract}{ In this article, we investigate the existence of  positive solutions to the following  class of quasilinear  {Schr\"odinger} equations involving Stein-Weiss type convolution
			\begin{equation*}
				\left\{
				\begin{array}{l}
					-\Delta_N u -\Delta_N (u^{2})u +V(x)|u|^{N-2}u=
					\left(\displaystyle\int_{\mb R^N}\frac{F(y,u)}{|y|^\beta|x-y|^{\mu}}~dy\right)\ds\frac{f(x,u)}{|x|^\beta} \; \text{\;\;\; in}\;
					\mathbb R^N,
				\end{array}
				\right. 
			\end{equation*}
			where $N\geq 2,$ $0<\mu<N,\, \beta\geq 0,$ and $2\beta+\mu< N.$ The potential  $V:\mb R^N\to \mb R$ is a continuous function satisfying $0<V_0\leq V(x)$ for all $x\in \mb R^N$ and some suitable assumptions. The nonlinearity  $f:\mb R^N\times \mb R\to \mb R$ is a continuous function with critical exponential growth in the sense of the Trudinger-Moser inequality and $F(x,s)=\int_{0}^s f(x,t)dt$ is the primitive of  $f$. 	 }
	\end{abstract}
	
	\maketitle
	\tableofcontents
	
		\section{Introduction and the main results}
	\noi This article concerns the existence of positive solutions for the following family of quasilinear Schr\"odinger equations with Stein-Weiss type convolution:
	\begin{equation*}
		\label{pq}\tag{$P_{*}$}\left\{
		\begin{array}{l}
			-\Delta_N u -\Delta_N (u^{2})u +V(x)|u|^{N-2}u=
			\left(\displaystyle\int_{\mb R^N}\frac{F(y,u)}{|y|^\beta|x-y|^{\mu}}~dy\right)\ds\frac{f(x,u)}{|x|^\beta} \; \text{in}\;
			\mathbb R^N,
		\end{array}
		\right.
	\end{equation*}
	where $\beta, \mu$ satisfy the following:
	\begin{align}
	\label{be}N\geq 2,\,\beta\geq 0,\,  0<\mu<N,\text{ and}\; 2\beta+\mu< N.
\end{align} 
The nonlinearity  $f:\mb R^N\times \mb R\to \mb R$ is a continuous function with critical exponential growth in the sense of the Trudinger-Moser inequality and fulfills some appropriate hypotheses, described later and $F(x,s)=\int_{0}^s f(x,t)dt$ is the primitive of  $f$. { The potential function  $V:\mb R^N\to \mb R$ is continuous and is assumed to  satisfy some suitable assumptions which are stated afterwards.} \\ 
	
	\noi  The problems driven by the quasilinear  operator $-\Delta u -\Delta (u^{2})u,$ have an ample amount of applications in the modeling of the physical phenomenon such as   dissipative quantum mechanics \cite{has}, plasma physics and fluid mechanics \cite{bass}, etc.
	Solutions of such problems are related to the existence of standing wave solutions for quasilinear Schr\"odinger equations of
	the form
	\begin{align}\label{l}
		iu_t =-\Delta u+V(x)u-h_1(|u|^2)u-C\Delta h_2(|u|^2)h_2'(|u|^2)u,\;\; x\in\mathbb R^N,
	\end{align}
	where $V$ is a continuous potential, $C$ is  some real constant, $h_1$ and $h_2$ are some real valued functions with some suitable assumptions. For  different types of    $h_2$, the quasilinear equations of the form \eqref{l} represent  different phenomenon in the mathematical physics. If $h_2(s)=s$ (see \cite{1}), then \eqref{l} models the superfluid film equation in plasma physics and in case of $h_2=\sqrt{1+s^2}$ (see \cite{33}), \eqref{l} describes the self-channeling of a high-power ultra short laser in matter.
	Observe that the term $\Delta_N(u^2)u$, present in problem  \eqref{pq}, restrains  the natural energy functional corresponding to the
	problem \eqref{pq} to be  well defined for all $u\in W^{1,N}({\mathbb R^N})$ (defined in Section \ref{sec2}). Hence, the variational method can't be applied directly for such problems of type \eqref{pq}. To deal with this 
	inconvenience, researchers have developed several methods and arguments, such as   a constrained minimization
	technique (see for e.g., \cite{20,22,29}), the perturbation method (see for e.g.,\cite{19,24})  and a
	change of variables (see for e.g., \cite{CJ}, \cite {do,8,12}).\\

	The nonlinearity in the problem \eqref{pq}   is nonlocal in nature. It is basically driven by the doubly weighted Hardy-Littlewood-Sobolev inequality (also called Stein-Weiss type inequality) and the Trudinger-Moser inequality in $\mb R^N$. Let us first recall the following doubly weighted Hardy-Littlewood-Sobolev inequality (see \cite{SW}).
	\begin{proposition}\label{HLS}
		(\textbf {Doubly Weighted Hardy-Littlewood-Sobolev inequality}) Let $t$, $s>1$ and $0<\mu<N $ with $\vartheta+\beta\geq0$, $\frac 1t+\frac {\mu+\vartheta+\beta}{N}+\frac 1s=2$, $\vartheta<\frac{ N}{t'}, \beta<\frac {N}{s'}$, $g_1 \in L^t(\mathbb R^N)$ and $g_2 \in L^s(\mathbb R^N)$, where $t'$ and $s'$ denote the H\"older conjugate of $t$ and $s$, respectively. Then there exists a  constant $C(N,\mu,\vartheta,\beta,t,s)$, independent of $g_1,$ $g_2$ such that
		\begin{equation}\label{HLSineq}
			\int_{\mb R^N}\int_{\mb R^N} \frac{g_1(x)g_2(y)}{|x-y|^{\mu}|y|^\vartheta|x|^\beta}\mathrm{d}x\mathrm{d}y \leq C(N,\mu,\vartheta,\beta,t,r)\|g_1\|_{L^t(\mb R^N)}\|g_2\|_{L^s(\mb R^N)}.
		\end{equation}
	
	\end{proposition}
	 	\noi	For $\vartheta=\beta=0$, it is reduced to the Hartee type (also called the Choquard type)  nonlinearity, which is driven by the classical Hardy-Littlewood-Sobolev inequality (see \cite{lieb}). So, the problem \eqref{pq} is closely related to the Choquard type equations. The study of Choquard equations started 
	with the seminal work of S. Pekar \cite{pekar}, where the author considered 
	the following nonlinear Schr\"{o}dinger-Newton equation:
	\begin{align}\label{sn}
		-\Delta u + V(x)u = ({\mathcal{K}}_\mu * u^2)u +\la f_1(x, u)\;\;\; \text{in}\; \mathbb R^N,
	\end{align}
	where $\la>0,$ $\mathcal{K}_\mu$ denotes  the  Riesz potential, $V:\mb R^N\to \mb R$ is a continuous function and $f_1:\mb R^N\times\mb R\to \mb R$ is a Carath\'eodory function with some appropriate growth assumptions.
	Researchers pay serious attention to studying such equations due to the rich applications in Physics, such as  the Bose-Einstein
	condensation (see \cite{bose}), the self gravitational
	collapse of a quantum mechanical wave function (see \cite{penrose}), etc. To describe the quantum theory of a polaron at rest and for modeling the phenomenon when an electron gets trapped in its own hole in the Hartree-Fock theory, P. Choquard (see \cite{choq}) used such elliptic equations of type \eqref{sn}.
	 For more rigorous  study  of Choquard type  equations, the readers can refer to
	\cite{lieb,Lions,moroz,moroz-main,moroz4,moroz5} and the references therein. \\
	
 One of the  main features of problem \eqref{pq} is that the 
	nonlinear term $f(x,s)$ has the maximal growth on $s$, that is, critical exponential growth in the sense of  Trudinger–Moser inequality in $\mb R^N$. For some bounded domain $\Om\subset \mb R^N$ ($N\geq 2$)   the  Trudinger–Moser inequality is established in \cite{Trud-Moser}. 
	That  form  of the Trudinger–Moser
	inequality is not valid in unbounded domains. Hence, for the whole of $\mb R^N,$ Cao \cite{cao} proposed an alternative version for the case $N=2$.  Later,
	for general  $N \geq 2$ such inequality is established by J. M. do \'O \cite{17}:
	\begin{theorem}\label{TM-ineq}(\textbf {Trudinger–Moser inequality in $\mb R^N$})
		If $\al>0,$ $N\geq 2$ and $u \in W^{1,N}({\mathbb R^N})$ then
		\[\int_{\mathbb R^N} \left[\exp(\alpha|u|^{\frac{N}{N-1}})-S_{N-2}(\al,u)\right]~dx < \infty,\] where $$\ds S_{N-2}(\al,u)=\ds\sum_{m=0}^{N-2}\frac{\al^m|u|^{\frac{m N}{N-1}}}{m!}.$$ Moreover, if $\alpha < \alpha_N$, and $
		\|u\|_{L^N(\mb R^N)}\leq K$, then there exists a positive constant $C=C(\al,K,N)$ such that
		\[\sup_{\|\nabla u\|_{L^N(\mb R^N)}\leq 1}\int_{\mathbb R^N} \left[\exp(\alpha|u|^{\frac{N}{N-1}})-S_{N-2}(\al,u)\right]~dx < C,\] where $\alpha_N = N\omega_{N-1}^{\frac{1}{N-1}}$ and $\omega_{N-1}=$ $(N-1)$- dimensional surface area of $\mb S^{N-1}.$
	\end{theorem}
	\noi  Here  $W^{1,N}({\mathbb R^N})$ is the Sobolev space, defined in Section \ref{sec2} and $L^N(\mb R^N)$ denotes the classical Lebesgue space. 
 The critical growth non-compact problems associated to this inequality in bounded domains are initially studied by  Adimurthi \cite{adi} and de Figueiredo et al. \cite{DR}. We would like to point out the fact that in the case of critical exponential problem involving $N$-Laplacian,   the critical exponential growth is equivalent to $\exp(|u|^{N/(N-1)})$. But in  problem \eqref{pq}, due to the term $\Delta_N (u^2)u$, the nature of the critical exponential growth is of the form $\exp(|s|^{2{N/(N-1)}})$. In the case $1<p<N$, the nonlinearity is of polynomial growth and the critical growth is equivalent to $ |u|^{2p^*},$ where  $p^*=Np/(N-p)$ (see for e.g. \cite{deng,8,27}).
In the same spirit, in the critical dimension  $N=2$,  the critical nonlinearity is expected to behave like $\exp(\al s^{4})$ as $s\rightarrow \infty$ (see \cite{do, wang}). 
For more development on this topic,  we refer to some recent contemporary works \cite{sv-2,sv-3}, where the authors studied the equations of type \eqref{pq} with critical exponential nonlinearity  for $N=2,$ without the convolution term $\ds\int_{\mb R^N}F(y,u)|x-y|^{-\mu}~dy$.\\

 For the case $\beta=0$ and $N=2$, without the quasilinear term $\De_N(u^2) u$ in  problem \eqref{pq},  Alves et al. \cite{yang-JDE} studied  the  following Choquard equation:
\[-\e^2\Delta u+V(x)u=\left(\displaystyle\int_{\mb R^2}\frac{F(y,u)}{|x-y|^{\mu}}~dy\right)f(x,u),\;\;\; x\in\mb R^2,\] where  $\e>0$, $0<\mu<2$, $V:\mb R^2\to\mb R$ is a continuous potential function with some suitable properties and the continuous function $f:\mb R^2\times\mb R\to\mb R$ enjoys the  critical exponential growth, in the sense of Trudinger-Moser inequality in $\mb R^2.$
In case of higher critical dimension, that is for $N\geq2$ and in  bounded domain $\Om\subset \mb R^N$, Giacomoni et al. \cite{AGMS} investigated the Kirchhoff-Choquard problems involving the $N$-Laplacian ($N\geq 2$) with critical exponential growth.
Furthermore,  the authors in \cite{bis} discussed the existence result for the problem \eqref{pq} in case of $\beta=0, V\equiv 0$ with critical exponential growth on the nonlinearity in a bounded domain  $\Om\subset \mb R^N,\, N\geq 2.$ \\

 {On the other hand, for the problems involving Stein-Weiss type nonlinearity, we can find very few  works
on that topic. In \cite{jjj}, Giacomoni et al. studied the polyharmonic Kirchhoff equations involving the singular weights with critical Choquard (that is, critical Stein-Weiss) type exponential nonlinearity in a bounded domain in $\mb R^N$ for $N\geq2$. Then, in \cite{gao-yang-du}, the authors discussed the following equation in the whole of $\mb R^N$:
$$
\displaystyle-\Delta u
=\frac{1}{|x|^{\beta}}\left(\int_{\mathbb{R}^{N}}\frac{|u(y)|^{2_{\beta, \mu}^{\ast}}}{|x-y|^{\mu}|y|^{\beta}}dy\right)
|u|^{2_{\beta, \mu}^{\ast}-2}u,~~~x\in\mathbb{R}^{N},
$$ where the critical 
exponent $2_{\beta, \mu}^* := \frac{2N-(\mu+\beta)}{N-2}$ is due to the  weighted Hardy-Littlewood-Sobolev inequality \eqref{HLSineq} (with $\beta=\vartheta$) and Sobolev embedding. For more   works on this type of nonlinearity, we refer to \cite{s4,s2,s1,s3} and references there in without attempting to provide the complete list.} \\ 

 {Inspired  by all the aforementioned works, in this article, we investigate the existence results for the problem \eqref{pq}, involving the Stein-Weiss type convolution with critical exponential nonlinearity.
We would like to highlight that, to the best of our knowledge, there is no such existence result in the literature for problem \eqref{pq} even for the case when the  dimension $N=2$. Moreover,
our result in this article is new in some sense  without the presence of the convolution term in the right hand side in problem \eqref{pq}. Furthermore, the equation of type \eqref{pq}  without the term $\Delta_N(u^2)u$ has not been addressed yet. All these three cases are covered in this article for the first time.\\    

Now we consider the following hypotheses  on the continuous  function $f:{\mathbb R^N}\times\mb R\to\mb R$:
\begin{enumerate}
\item[$(f_1)$] $f\in C({{\mathbb R^N}}\times \mb R,\mb R)$ such that for a.e. $x\in\mb R^N,$ $f(x,s)=0,$ if $s\le 0$ and  $f(x,s)>0,$ if
$s>0$.
\item[$(f_2)$] $\displaystyle \lim_{s\to 0^+}\frac{f(x,s)}{s^{N-1} }=0$ uniformly in $x\in\mathbb R^N.$
\item[$(f_3)$]There exists some $\al_0>0$ such that, \begin{equation*}
\ds\lim_{s\to+\infty}f(x,s)\exp\left(-\al|s|^{\frac{2N}{N-1}}\right)=	\left\{
\begin{array}{l}
0,\;\;\;\;\;\;\;\;~~\mbox{ for all}~~\al>\al_0,\\
+\infty,~~\;\;\mbox{ for all}~~ \al<\al_0.
\end{array}
\right.
\end{equation*}
\item[$(f_4)$] There exist positive constants  $s_0$, $M_0$ and $m_0$ such that
\[ 0<s^{m_0} F(x,s)\le M_0 f(x,s),\;\mbox{for all}\; (x,s)\in \mb R^N\times[s_0,+\infty).\]
\item[$(f_5)$]  There exists $\ell>N$  such that $0<\ell F(x,s)\leq f(x,s)s,$ for all $ s>0.$
\item[$(f_6)$] We assume that\begin{equation}\label{h-growth}
\displaystyle \lim_{s\to +\infty} \frac{sf(x,s)F(x,s)}{\exp\left(2 |s|^{\frac{2N}{N-1}}\right)} = +\infty,\mbox{ uniformly in }x \in {{\mathbb R^N}}.
\end{equation}
\end{enumerate}
\begin{remark}
It is important to mention that the condition \eqref{h-growth} can be weakened by replacing $\infty$ by some real number $C$, whose optimum value can  be estimated  (see \cite{ay,M-do1}, etc).  But we would like to convey that  it doesn't  bring any effective
changes in this context of the main result except for the Lemma \ref{PS-level}.
\end{remark}
Next, we consider the following  assumptions on the potential function $V:$
\begin{enumerate}
\item[$(V_1)$] $V\in C({{\mathbb R^N}}, \mb R)$ such that $\displaystyle\inf_{x\in \mb R^N} V(x)=V_0>0$.
\item[$(V_2)$] The potential function $V:\mb R^N\to\mb R$ satisfies $\ds\lim_{|x|\to+\infty} V(x)=+\infty.$
\item[$(V_2^\prime)$] 
$\displaystyle V(x)\leq\lim_{|x|\to+\infty} V(x)=V_\infty<+\infty,$ with $V\not=V_\infty.$

\end{enumerate}
Finally, we state the main theorems  in this article.  
The first theorem in this article reads as:
\begin{theorem}\label{T2}
	Let $\beta, \mu$ satisfy \eqref{be}. Suppose that  the hypotheses $(V_1)$, $(V_2)$ and $(f_1)$-$(f_6)$ are satisfied. Then the problem \eqref{pq} has a nontrivial positive  solution.
\end{theorem}
\begin{remark}
In Theorem \ref{T2}, due to the presence of Stein-Weiss type nonlinearity, the corresponding energy functional is not translation invariant. Therefore, to show the existence of nontrivial solution, we are unable to use Lions type concentration compactness lemma, when we consider the potential function $V$ to be satisfying the condition $(V_1)$, $(V_2)$. We propose this case as an open question even for the problem \eqref{pq} without the quasilinear term $\De_N(u^2)u$. So to tackle this issue, we consider the assumption $(V_2^\prime)$ instead of $(V_2)$ in the above theorem.
\end{remark}
\noi In the next theorem, we deal with the Choquard type nonlinearity considering $(V_2^\prime)$ which is a weaker assumption on $V$ and creates a non compact situation while investigating  a  positive solution to the problem \eqref{pq}. 
\begin{theorem}\label{T3}
	Let  $\beta=0$ in \eqref{be}.  Assume that the hypotheses $(V_1)$, $(V_2^\prime)$ and $(f_1)$-$(f_6)$ hold. Then the problem \eqref{pq} has a nontrivial positive solution.
\end{theorem}
half of the critical level when the principal operator does not contain $\De_N(u^2)u$. 
The main  contributions in this work are as follows:
\begin{itemize}
\item Finding the suitable first critical energy level (which is exactly half of the critical level when the principal operator does not contain $\De_N(u^2)u$)  to describe the  convergence of the Cerami sequences below this level.
\item For the case $\beta=0$, constructing a suitable path between two energy functionals $J$ (see \eqref{energy}) and $I_\infty$ (see \eqref{iinft}) due to the lack of compactness occurred in the embedding $W^{1,N}(\mb R^N)$ into  $L^N(\mb R^N)$ (see Section \ref{sec2}) because of the condition $(V_2^\prime)$ as well as, due to the presence of the term $\De_N(u^2)u$ in our the problem.
\item Establishing a general Pohozaev type identity related to our problem which is an important tool to prove Theorem \ref{T3}.
\end{itemize}   To analyze these, we need to carry out very delicate and crucial estimates.
\begin{remark}
Following the similar idea, we can establish  the  results as in  Theorem \ref{T3}  under any of the following conditions imposed on the potential functions $V$:
\begin{itemize}
\item[$1.$] (Compact-coercive case):
$V$ satisfies $(V_1)$ and $(V_2)$. 
In this case, we obtain the compact embedding from $W^{1,N}(\mb R^N)$ into $L^q(\mb R^N)$ for $N\leq q<\infty$.
\item[$2.$] (Radially symmetric case):
$V$ satisfies $(V_1)$ and 
$ V(x)=V(|x|)$ for all $x\in \mb R^N$.
In this case, we have the compact embedding from $W^{1,N}(\mb R^N)$ into $L^q(\mb R^N)$ for $N< q<+\infty$.
\item[$3.$] (Asymptotic case of a periodic
function): $V$ satisfies $(V_1)$ with $\ds\lim_{|x|\to+\infty} V(x)=V_m(x), $ where $V_m$ is a $1$-periodic continuous function. Also, $V(x)\leq V_m(x)$ for all $x\in\mb R^N$ and $V(x)<V_m(x)$ on a positive Lebesgue measure set of $\mb R^N$. 
\end{itemize}

\end{remark}
\noi 
\textbf{Notation.}  Throughout this paper, we make use of the following notations:
\begin{itemize}
\item $C_{1},C_{2},\cdots, \tilde C_1, \tilde C_2,\cdots, C$ and $\tilde C$ denote (possibly different from line to line) positive constants.
\item For any exponent $p>1,$ $p'$ denotes the conjugate of $p$ and is given as $p'=\frac{p}{p-1}.$
\item $B_r(x)$ denotes the ball of radius $r$ centered at $x\in\mb R^N$.   
\item If $S$ is a measurable set in $\mathbb{R}^{N}$, we denote the Lebesgue measure of $S$  by $\vert S \vert$ . 
\item The arrows $\rightharpoonup $ and $\to $ denote the weak convergence and  strong convergence, respectively.
\end{itemize}

\section{Variational Frame-work}\label{sec2}
\noi In this section we recall some preliminary results.	For  $1\leq p < \infty$,  
the Sobolev space $W^{1,p}({\mathbb R^N})$ is defined as 
$$W^{1,p}({\mathbb R^N})=\left\{u\in L^p({\mathbb R^N}) : \int_{\mathbb R^N}|\nabla u|^p dx<\infty\right\}$$ which is a Banach space  equipped with the norm
$$\|u\|_{1,p}:=\displaystyle \left({\int_{{\mathbb R^N}} |\nabla u|^p dx}+\int_{\mathbb R^N } |u|^p dx\right)^{1/p}.$$
\noi When $p=N$, we denote $\|\cdot\|:=\|\cdot\|_{1,N}$.	The embedding $ W^{1,N}({\mathbb R^N}) \hookrightarrow L^q({\mathbb R^N})$ is continuous for  $N\leq q<+\infty.$ This embedding is compact for $N<q<+\infty.$ The space $(W^{1,N}(\mb R^N))^*$ is the topological dual of $W^{1,N}(\mb R^N)$.\\Next, when the potential function $V$ satisfies the assumptions $(V_1)$ and $(V_2)$, we define a  subspace  of $W^{1,N}(\mb R^N)$
\[E:=\left\{u\in W^{1,N}(\mb R^N)\,:\,  \ds\int_{\mb R^N} V(x)| u|^N dx<+\infty\right\}\] endowed with the norm 
\[\|u\|_E:=\displaystyle \left({\int_{{\mathbb R^N}} |\nabla u|^N dx}+\int_{\mathbb R^N } V(x)|u|^N dx\right)^{1/N}.\] Note that $(E,\|\cdot\|_E)$ is a Banach space and since $V(x)\geq V_0>0,$ it follows that $E$ is  continuously
embedded in  $W^{1,N}({\mathbb R^N})$. Moreover, there exist continuous embedding and compact embedding from  $ W^{1,N}({\mathbb R^N})$ to  $L^q({\mathbb R^N})$ for  $N\leq q<+\infty.$
The topological dual of the space $E$ is denoted by  $E^*$.  When $V$ satisfies $(V_1)$ and $(V_2^\prime)$ then $\|\cdot\|_E$ is equivalent to $\|\cdot\|$, that is, $E=W^{1,N}(\mb R^N)$. Also, note that $C_c^\infty(\mb R^N)$ is dense in $E$ and in $W^{1,N}(\mb R^N)$ with respect to corresponding norms.\\
\noi Due to the presence of the singular weight $|x|^{-\beta}$ in the problem \eqref{pq}, we recall the following version of the Trudinger-Moser inequalities studied by Adimurthi and Sandeep \cite{as} for bounded domains  $\Om\subset \mb R^N$.
\begin{theorem}\label{as}  Let $\Om\subset\mb R^N (N\geq2)$ be a smooth bounded domain. Then for $u\in W^{1,N}_0(\Om)$ and for any $\al>0,\; 0<\beta<N,$
$$\frac{\exp\left(\al |u|^{\frac{N}{N-1}}\right)}{|x|^\beta}\in L^1(\Om).$$ Moreover,
\[\sup_{\|u\|_{W^{1,N}_0(\Om)}\leq 1}\int_\Om \frac{\exp(\alpha|u|^{\frac{N}{N-1}})}{|x|^\beta}~dx < \infty\]
if and only if $\alpha/\al_N+\beta/N \leq 1$, where \[W^{1,N}_0({\Om})=\left\{u\in W^{1,N}(\Om) :  u=0 \text{\;\; on }\; \partial\Om\right\}\] with the norm 
\[ \|u\|_{W^{1,N}_0(\Om)}:=\displaystyle \left({\int_{{\Om}} |\nabla u|^N dx}\right)^{1/N}.\]
\end{theorem}
\noi The analogous version of this result in the whole of $\mb R^N$ was established by Adimurthi and Yang \eqref{ay} and we state that result  by taking $\tau=1$. 
\begin{theorem} \label{ay}    For any $\al>0,\; 0\leq\beta<N$ and $u\in W^{1,N}(\mb R^N)$, there holds
$$\frac{\exp\left(\al |u|^{\frac{N}{N-1}}-S_{N-2}(\al,u)\right)}{|x|^\beta}\in L^1(\mb R^N).$$ Moreover
\[\sup_{\|u\|\leq 1}\int_{\mathbb R^N} \frac{\exp(\alpha|u|^{\frac{N}{N-1}})-S_{N-2}(\al,u)}{|x|^\beta}~dx < \infty\]
if and only if $\frac{\alpha}{\al_N}+\frac{\beta}{N} \leq 1$.\end{theorem}
{\noi	The  natural energy functional  associated to the problem \eqref{pq} is the following:
\begin{align*}I(u)=\frac{1}{N}\displaystyle\int_{{\mathbb R^N} }(1+2^{N-1}|u|^{N})|\nabla u|^{N}dx+\frac 1N\int_{\mb R^N}V(x)|u|^N dx-\frac 12\int_{{\mathbb R^N}}\left(\int_{{\mathbb R^N}} \frac{F(y,u(y))}{|y|^\beta|x-y|^\mu}dy\right)\frac{F(x,u(x))}{|x|^\beta}dx.
\end{align*}	Note that  the term  $\displaystyle\int_{{\mathbb R^N}} u^{N}|\nabla u|^{N}dx$ is not finite for all $u\in W^{1,N}({\mathbb R^N}) $. Thus,  the functional $I$ is not well defined in $W^{1,N}({\mathbb R^N})$.}
To overcome this difficulty, we employ the following change of variables which was introduced in \cite{CJ}, namely, $w:=h^{-1}(u),$ where $h$ is defined by
\begin{equation}\label{g}
\left\{
\begin{array}{l}
h^{\prime}(s)=\displaystyle\frac{1}{\left(1+2^{N-1}|h(s)|^{N}\right)^{\frac{1}{N}}}~~\mbox{in}~~ [0,\infty),\\
h(s)=-h(-s)~~\mbox{in}~~ (-\infty,0].
\end{array}
\right.
\end{equation}
\noi	Now we state  some important and useful properties of $h$. For the detailed proofs of such results, one can see \cite{CJ,do} and references there in.
\begin{lemma}\label{L1} The function $h$ satisfies the following properties:
\begin{itemize}
\item[$(h_1)$] $h$ is uniquely defined, $C^{\infty}$ and invertible;
\item[$(h_2)$] $h(0)=0$;
\item[$(h_3)$] $0<h^{\prime}(s)\leq 1$ for all $s\in \mathbb{R}$;
\item[$(h_4)$] $\frac{1}{2}h(s)\leq sh^{\prime}(s)\leq h(s)$ for all $s>0$;
\item[$(h_5)$] $|h(s)|\leq |s|$ for all $s\in \mathbb{R}$;
\item[$(h_6)$] $|h(s)|\leq 2^{1/{(2N)}}|s|^{1/2}$ for all $s\in \mathbb{R}$;
\item[$(h_7)$] $\ds \lim_{s\to+\infty}\frac {h(s)}{s^{\frac 12}}=2^{\frac{1}{2N}}$;
\item[$(h_8)$]  $|h(s)|\geq h(1)|s|$ for $|s|\leq 1$ and $|h(s)|\geq h(1)|s|^{1/2}$ for $|s|\geq 1$;
\item[$(h_9)$] $h^{\prime \prime}(s)<0$ when $s>0$ and $h^{\prime \prime}(s)>0$ when $s<0$.
\item[$(h_{10})$] $\displaystyle\lim_{s\to 0}\frac{h(s)}{s}=1.$ 
\end{itemize}
\begin{example}
One of the examples of such functions is given in 
\begin{figure}[H]
\includegraphics[width=80mm, height=50mm]{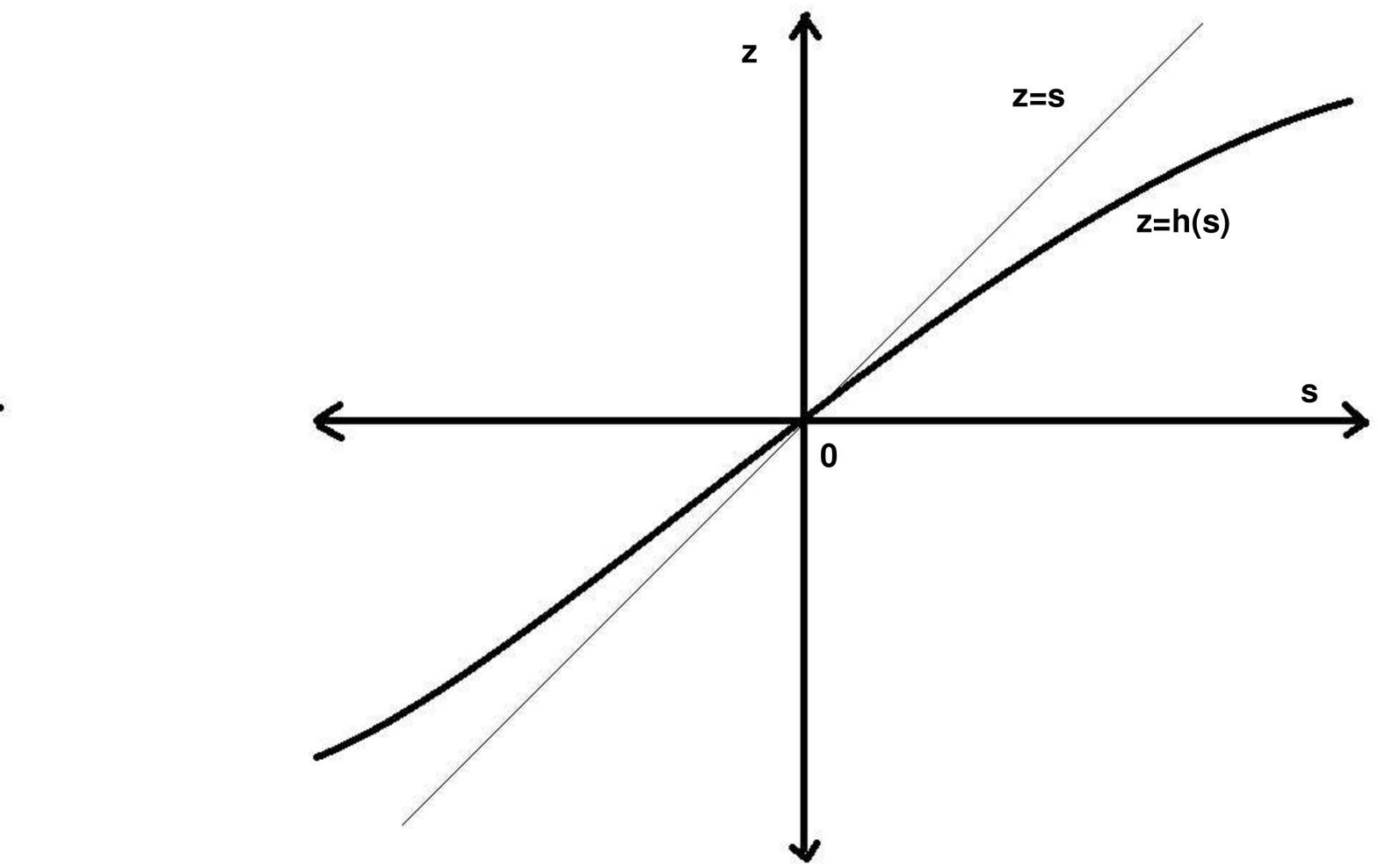}
\caption{Plot of the function $h$}
\label{pic}
\end{figure}	
\end{example}
\end{lemma}
{\noi	After applying the change of variable  $w=h^{-1}(u),$ we define the new functional 
\begin{align}\label{energy}J(w)=\frac{1}{N}\displaystyle\int_{{\mathbb R^N}} |\nabla w|^{N}dx+\frac 1N \int_{{\mathbb R^N} } V(x)|h(w)|^N dx-\frac 12\int_{{\mathbb R^N}}\left(\int_{{\mathbb R^N}} \frac{F(y,h(w))}{ |y|^\beta|x-y|^\mu}dy\right)\frac{F(x,h(w))}{|x|^\beta}dx
\end{align}
for $w$ in some appropriate Sobolev space ($ W^{1,N}({\mathbb R^N}) \text{ or } E$, as per the convenience).}
Now using Lemma \ref{L1}-$(h_5)$, we get \begin{equation}\label{go}
	\int_{\mb R^N} V(x)	|h(w)|^N dx\leq\int_{\mb R^N} V(x)	|w|^N dx<\left\{
	\begin{array}{l}
		+\infty~~\mbox{ on}\;\;\; E,\;\;\;\;\;\;\;\;\;\;\;\;\;\;\;\; \text{ if $V$ satisfies $(V_2)$},\\
		+\infty~~\mbox{ on}\;\;\; W^{1,N}({\mathbb R^N}),\;\; \text{ if $V$ satisfies $(V_2^\prime)$}.
	\end{array}
	\right.
\end{equation}
\noi Moreover, from the assumptions $(f_2)$-$(f_3)$, we obtain that for any $\e>0$, $r\geq N$,  there exist  positive constants $ C(r,\e)>0$, $\al>\al_0>0$  such that
\begin{align} 
|F(x,s)| \le \e |s|^N + C(r,\e) |s|^r \left[\exp\left(\al|s|^{\frac{2N}{N-1}}\right)-S_{N-2}(\al, s^2)\right]\;\; \text{for all}\; (x,s)\in {\mathbb R^N} \times \mb R.\label{k1}
\end{align}
Thus, in light of the Sobolev embedding, for any $v \in W^{1,N}({\mathbb R^N})$,
\begin{equation}\label{KC-new1}
F(x,v) \in L^{q}({\mathbb R^N})\mbox{ for any }q\geq N.
\end{equation} Now by \eqref{g}, if $w\in W^{1,N}({\mathbb R^N}),$ then $h(w)\in W^{1,N}({\mathbb R^N})$. Therefore,
Proposition \ref{HLS}  with $t = s$ and  $\vartheta=\beta (\geq 0)$ implies that
\begin{align}\label{k2}
\int_{{\mathbb R^N}}\left(\int_{{\mathbb R^N}}\frac{F(y,h(w))}{|y|^\beta|x-y|^{\mu}}dy\right)\frac{F(x,h(w))}{|x|^\beta}dx  \leq C(N,\mu,\beta){\|F(\cdot,h(w))\|^2_{L^{\frac{2N}{{2N-2\beta-\mu}}}({\mathbb R^N})}}.
\end{align} This together with the properties of $f$ and $h$, \eqref{go} and \eqref{KC-new1} yields that $J$ is well defined  and $J$ is a  $C^1$ functional.
{A function $w \in W^{1,N}({\mathbb R^N})$ is a critical point of the functional $J$,  if for every $v \in  W^{1,N}({\mathbb R^N})$, we have
\begin{align}\label{weak}
&\int_{{\mathbb R^N}} |\nabla w|^{N-2}\nabla w\nabla v dx+\int_{\mb R^N} V(x)|h(w)|^{N-2}h(w)h'(w)v dx\notag\\&\qquad\qquad\qquad\qquad-\int_{{\mathbb R^N}}\left(\int_{{\mathbb R^N}} \frac{F(y,h(w))}{|y|^\beta|x-y|^\mu}dy\right)\frac{f(x,h(w))}{|x|^\beta}h'(w)v(x)dx=0.
\end{align}   
Therefore,    $w$ is a weak solution   to the following problem:
\begin{equation}\label{pp}
-\Delta_N w + V(x)|h(w)|^{N-2}h(w)h'(w)=
\left(\displaystyle\int_{\mb R^N}\frac{F(y,h(w))}{|y|^\beta|x-y|^{\mu}}~dy\right)\frac{f(x,h(w))} {|x|^\beta}h'(w) \; \text{in}\;
{\mathbb R^N}.
\end{equation}}
As in \cite{sv-1}, it can be verified  that the transformed problem \eqref{pp} is equivalent to the problem \eqref{pq}, which takes $u = h(w)$ as its solution. Thus,  it is enough to show  the existence of solution to  \eqref{pp} by finding the critical point of $J$, and then apply the transformation $h$ on that solution, which will serve as the solution to the problem \eqref{pq}.\\
{Next, we have the following two technical results for the later consideration.}
\begin{lemma}\label{mp-inq}
For any $p\ge 1, \al>0$, it holds that
\begin{align*}
\left(\exp\left(\al|s|^{\frac{N}{N-1}}\right)-S_{N-2}(\al, s)\right)^p\leq \exp\left(p\al|s|^{\frac{N}{N-1}}\right)-S_{N-2}(p\al, s), \text{\;\; for all $s\in \mb R.$}
\end{align*} \end{lemma}
\begin{proof}
The proof follows by analyzing the limits of the expression $$\ds\frac{\left(\exp\left(\al|s|^{\frac{N}{N-1}}\right)-S_{N-2}(\al, s)\right)^p}{\exp\left(p\al|s|^{\frac{N}{N-1}}\right)-S_{N-2}(p\al, s)},$$ as $|s|\to 0$ and $|s|\to+\infty$ 
using L'Hospital's rule.
\end{proof}
{\begin{lemma}\label{techn}
Let the function $h$ be defined as in \eqref{g}. Then for any $w\in W^{1,N}(\mb R^N)$ and $p>0$, we have
\[\exp\left(p|h(w)|^{\frac{2N}{N-1}}\right)-S_{N-2}\left (p, |h(w)|^2\right)\leq \exp\left(p2^{\frac{1}{N-1}}|w|^{\frac{N}{N-1}}\right)-S_{N-2}\left(p2^{\frac{1}{N-1}}, |w|\right).\]
\end{lemma}
\begin{proof}Recalling Lemma \ref{L1}-$(h_6)$,  we obtain
\begin{align*}\exp\left(p|h(w)|^{\frac{2N}{N-1}}\right)-S_{N-2}\left(p, |h(w)|^2\right)=\ds\sum_{m=N-1}^{\infty}\frac{p^m|h(w)|^{\frac{2m N}{N-1}}}{m!}&\leq \ds\sum_{m=N-1}^{\infty}\frac{\left( p 2^{\frac{1}{N-1}}\right)^m|w|^{\frac{m N}{N-1}}}{m!}.
\end{align*}Thus, the result is proved.
\end{proof}
\section{Analysis of the Cerami sequence and convergence results}
\noi In this section, we discuss the behaviour of the Cerami sequence of $J$ and prove some convergence results. The sequence $\{w_k\}$ in  $W^{1,N}({\mathbb R^N})$ is called a Cerami sequence of $J$ at level $c\in\mb R$, if
\begin{align*}
J(w_k) \to c\,; \text{\;\; and \;}  \;\left(1+\|w_k\|\right) J^{\prime}(w_k) \to 0 \text{\; in \;} (W^{1,N}({\mathbb R^N}))^*\text{\;\;\;  as $ k \to+\infty$.}
\end{align*}
\begin{lemma}\label{lem712}
Let \eqref{be} hold. Suppose  the function  $h$ is defined in \eqref{g} and let the hypotheses $(V_1)$, $(V_2^\prime)$ on $V$ and $(f_1)$ and $(f_5)$ on $f$ be satisfied. Then every Cerami sequence of $J$ is bounded in $W^{1,N}({\mathbb R^N})$. 
\end{lemma}
\begin{proof}
Let $\{w_k\} \subset W^{1,N}({\mathbb R^N})$ be a Cerami sequence of $J$ at level $c\in\mb R$.
Then, as $k \to+\infty$, we have
\begin{align}\label{kc-PS-bdd0}
J(w_k)=	&\frac 1N \int_{\mb R^N}|\nabla w_k|^N dx+\frac 1N \int_{\mb R^N} V(x) |h(w_k)|^N dx - \frac12 \int_{{\mathbb R^N}} \left(\int_{{\mathbb R^N}} \frac{F(y,h(w_k))}{|y|^\beta|x-y|^{\mu}}dy \right)\frac{ F(x,h (w_k))}{|x|^\beta}~dx \to c,
\end{align} and for any $\phi\in W^{1,N}(\mb R^N)$, 
\begin{align}\label{kc-PS-bdd1}
&\left(1+\|w_k\|\right)	\left|  \langle J^{\prime}(w_k),\phi\rangle \right|\notag\\&=\left(1+\|w_k\|\right)	\left| \int_{{\mathbb R^N}}{|\nabla w_k|^{N-2}\nabla w_k \nabla\phi}{dx}+\int_{\mb R^N} V(x) |h(w_k)|^{N-2} h(w_k)h'(w_k)\phi dx\right.\notag\\
&\quad\quad\quad\quad\quad\quad\quad\quad\quad\left. -\int_{\mathbb R^N} \left(\int_{\mathbb R^N} \frac{F(y,h(w_k))}{|y|^\beta|x-y|^{\mu}}dy \right)\frac{f(x,h(w_k))}{|x|^\beta} h'(w_k)\phi ~dx \right|\leq\e_k\|\phi\|,
\end{align}
where $\e_k \to 0$ as $k\to+\infty$.  In the last relation, taking $\phi=w_k$, we get 
\begin{equation}\label{kc-PS-bdd2}
\left| \int_{\mb R^N}|\nabla w_k|^{N}dx+\int_{\mb R^N} V(x) |h(w_k)|^{N} dx-\int_{\mathbb R^N} \left(\int_{\mathbb R^N} \frac{F(y,h(w_k))}{|y|^\beta|x-y|^{\mu}}dy \right)\frac{f(x,h(w_k)}{|x|^\beta} h'(w_k) w_k ~dx \right|\leq \e_k.
\end{equation}
Now set $$v_k:=\frac{h(w_k)}{h'(w_k)}.$$  By Lemma \ref{L1}-$(h_4)$, we have $|v_k|\leq 2 |w_k|.$  Moreover,  using \eqref{g} and Lemma \ref{L1}-$(h_5)$, we get 
$$|\nabla v_k|^N= \left(1+\frac{2^{N-1}|h(w_k)|^N}{1+2^{N-1}|h(w_k)|^N}\right)|\nabla w_k|^N\leq 2|\nabla w_k|^N. $$	 Therefore, 
$$\|v_k\|\leq2\|w_k\|.$$ Therefore, $v_k\in W^{1,N}({\mathbb R^N})$ and hence, in \eqref{kc-PS-bdd1}, taking $\phi=v_k$, it follows that
\begin{align}\label{1.0}
\left|\langle J'(w_k), v_k\rangle\right|&=\bigg|\int_{\mathbb R^N}\left(1+\frac{2^{N-1}|h(w_k)|^N}{1+2^{N-1}|h(w_k)|^N}\right)|\nabla w_k|^N dx+\int _{\mb R^N}V(x)|h(w_k)|^{N} dx\notag\\&\;\;\;\;\;\;-\int_{\mathbb R^N} \left(\int_{\mathbb R^N} \frac{F(y,h(w_k))}{|y|^\beta|x-y|^{\mu}}dy \right)\frac{f(x,h(w_k))} {|x|^\beta}h(w_k) dx\bigg|\notag\\
&\leq\e_k\frac{\|v_k\|}{(1+\|w_k\|)}\leq 2\e_k.
\end{align}	
Recalling \eqref{kc-PS-bdd0}, \eqref{1.0}  and $(f_5)$, we get
\begin{align}\label{bdd0}
c+\frac {\e_k}{\ell}&\geq J(w_k)-\frac{1}{2\ell}\langle J'(w_k), v_k\rangle\notag\\
&\geq \int_{\mathbb R^N}\left[\frac1N-\frac{1}{2\ell}\left(1+\frac{2^{N-1}|h(w_k)|^N}{1+2^{N-1}|h(w_k)|^N}\right)\right]|\nabla w_k|^N dx+\int_{\mb R^N} \left(\frac1N-\frac{1}{2\ell}\right)V(x)|h(w_k)|^N dx\notag\\&\;\;\;\;\;\;-\frac12\int_{\mathbb R^N} \left(\int_{\mathbb R^N} \frac{F(y,h(w_k))}{|y|^\beta|x-y|^{\mu}}dy \right)\left[\frac{F(x,h(w_k))}{|x|^\beta}-\frac1\ell \frac{f(x,h(w_k))}{|x|^\beta} h(w_k)\right] dx\notag\\
&\geq \int_{\mathbb R^N}\left(\frac1N-\frac{1}{\ell}\right)|\nabla w_k|^N dx+\int_{\mb R^N} \left(\frac1N-\frac{1}{2\ell}\right)V(x)|h(w_k)|^N dx.
\end{align}\label{bdd2}
This yields that there is a constant $C$ (independent of $k, w_k, v_k$ ) such  that \begin{align}\label{bdd1}\int _{\mb R^N}|\nabla w_k|^Ndx+\int_{\mb R^N}V(x)|h(w_k)|^Ndx<C,
\end{align} since $\ell>N$. 
{Using \eqref{bdd1} together with  Lemma \ref{L1}-$(h_5)$,$(h_8)$, $(V_1)$ and the Gagliardo-Nirenberg interpolation inequality, we deduce
\begin{align*}
\|w_k\|_{L^N(\mb R^N)}^N&=	\int_{\mb R^N} |w_k|^N dx=\int_{\{w_k\leq 1\}} |w_k|^N dx+\int_{\{w_k>1\}} |w_k|^N dx\\
&\leq \frac {1} { V_0(h(1))^N}\int_{ \{w_k\leq 1\}} V(x) |h(w_k)|^N dx+\frac {1} {(h(1))^N}\int_{\{w_k>1\}} |h(w_k)|^{2N}dx\\
&\leq C+C \|\nabla h(w_k) \|_{L^N{(\{w_k>1\})}}^N\| h(w_k) \|_{L^N{(\{w_k>1\})}}^N\\
&\leq C+\frac{ C}{{ V_0}}\|h'(w_k)\nabla w_k \|_{L^N{(\mb R^N)}}^N  \int_{\mb R^N}V(x)|h(w_k)|^N dx \leq C.		
\end{align*}}
\noi From the last estimate, we  infer that the sequence $\{w_k\}$ is bounded in $W^{1,N}(\mb R^N)$. This concludes the lemma. 
\end{proof}\noi Let $\{w_k\}\subset W^{1,N}({\mathbb R^N})$ be a Cerami sequence for $J$.
Then Lemma \ref{lem712} yields that $\{w_k\}$  is bounded in $W^{1,N}({\mathbb R^N})$. Thus, there exists  $w\in W^{1,N}({\mathbb R^N})$ such that up to a subsequence, still denoted by $\{w_k\}$, as $k \to+\infty$
\begin{equation}\label{cnvv}\left.
\begin{split}&w_k \rightharpoonup w \text { \;\;\; weakly in } W^{1,N}({\mathbb R^N}),\\
&w_k \to w \text{\;\; strongly in } L_{loc}^q({\mathbb R^N}),\,\text{for all \;\;} q \in [1,\infty),\\
&w_k(x) \to w(x)\text{ \;\;\; point-wise a.e. in } {\mathbb R^N}.\end{split}\right\}
\end{equation}  Then  from \eqref{kc-PS-bdd1} and \eqref{kc-PS-bdd2}, we deduce  
\begin{align}
\int_{\mathbb R^N} \left(\int_{\mathbb R^N}\frac{F(y,h(w_k))}{|y|^\beta |x-y|^\mu}dy\right)\frac{F(x,h(w_k))}{|x|^\beta}~dx &\leq C,\label{wk-sol10}\\
\int_{\mathbb R^N} \left(\int_{\mathbb R^N}\frac{F(y,h(w_k))}{|y|^\beta|x-y|^\mu}dy\right)\frac{f(x,h(w_k))}{|x|^\beta}h'(w_k)w_k~dx & \leq C,\label{wk-sol100}
\end{align} where $C$ is some positive constant. 
Now in the next  lemmas,  we consider the Cerami sequence $\{w_k\}$ to be satisfying all the   facts stated above.
{\begin{lemma}\label{PS-ws} Assume that \eqref{be} holds. Let the function $h$ be defined as in \eqref{g}	and  the assumptions $(f_1)$-$(f_5)$ hold. Suppose $\{w_k\}\subset W^{1,N}({\mathbb R^N})$ is a Cerami sequence for $J$.
Then for any $\Om\Subset\mb R^N,$ it follows that	\begin{align*}
\lim_{k \to+\infty} \int_{{\Om}} \left(\int_{{\mathbb R^N}}\frac{F(y,h(w_k))}{|y|^\beta|x-y|^{\mu}} dy \right) \frac{F(x,h(w_k))}{|x|^\beta}  ~dx = \int_{{\Om}} \left(\int_{{\mathbb R^N}}\frac{F(y,h(w))}{|y|^\beta|x-y|^{\mu}} dy \right) \frac{F(x,h(w))}{|x|^\beta}~dx.
\end{align*} 
\end{lemma}}
\begin{proof}
Since $\{w_k\}$ is a Cerami sequence, it is bounded in  $W^{1,N}(\mb R^N)$ by Lemma \ref{lem712} and it verifies \eqref{cnvv}.
On the other hand, from \eqref{k2}, we have
$$\left(\displaystyle\int_{\mathbb R^N}\frac{F(y,h(w))}{|y|^\beta|x-y|^\mu}dy\right)\frac{F(\cdot,h(w))}{|x|^\beta} \in L^{1}({\mathbb R^N}).$$
Therefore, for any $\hat \de>,0$  we can choose $R> \max\left\{1,\left(\frac{2C M_0}{{\hat\de}}\right)^{\frac{1}{m_0+1}}, s_0\right\}$, where $C$ is defined in \eqref{wk-sol10} and \eqref{wk-sol100}, such that
{\begin{align}\label{bb}
&\int_{\{h(w)\geq R\}} \left(\int_{{\mathbb R^N}} \frac{F(y,h(w))}{|y|^\beta|x-y|^{\mu}} dy \right)\frac{F(x,h(w))}{|x|^\beta}  ~dx  \leq \hat\de.
\end{align}}
Now using $(f_4)$, Lemma \ref{L1}-$(h_4)$ and \eqref{wk-sol100}, we obtain
\begin{align}\label{b}
&\int_{\{h(w_k)\geq R\}} \left(\int_{{\mathbb R^N}} \frac{F(y,h(w_k))}{|y|^\beta|x-y|^{\mu}} dy \right)\frac{ F(x,h(w_k))} {|x|^\beta} ~dx \notag\\
&\leq M_0 \int_{\{h(w_k)\geq R\}} \left(\int_{{\mathbb R^N}} \frac{F(y,h(w_k))}{|y|^\beta|x-y|^{\mu}}dy\right) \frac{f(x,h(w_k))h(w_k)}{|x|^\beta(h(w_k))^{m_0+1}}   ~dx\notag\\
&\leq \frac{2M_0}{ R^{m_0+1}}\int_{ \{h(w_k)\geq R\}} \left(\int_{{\mathbb R^N}} \frac{F(y,h(w_k))}{|y|^\beta|x-y|^{\mu}} dy\right) {\frac{f(x,h(w_k))} {|x|^\beta}h'(w_k)w_{k}} ~dx<\hat\de.
\end{align}
Gathering  \eqref{bb} and \eqref{b}, we get
\begin{align*}
&\left| \int_{{\Om}} \left(\int_{{\mathbb R^N}}\frac{F(y,h(w_k))}{|y|^\beta|x-y|^{\mu}} dy \right) \frac{F(x,h(w_k))} {|x|^\beta}~dx-  \int_{{\Om}} \left(\int_{{\mathbb R^N}}\frac{F(y,h(w) )}{|y|^\beta|x-y|^{\mu}} dy \right) \frac{F(x,h(w))}{|x|^\beta} ~dx\right|\\
&\leq 2\hat\de+ \bigg|\int_{\Om\cap \{h(w_k)\leq R\}} \left(\int_{{\mathbb R^N}}\frac{F(y,h(w_k))}{|y|^\beta|x-y|^{\mu}} dy \right) \frac{F(x,h(w_k))}{|x|^\beta} ~dx\\&\qquad\qquad\qquad- \int_{\Om\cap \{h(w)\leq R\}} \left(\int_{{\mathbb R^N}}\frac{F(y,h(w))}{|y|^\beta|x-y|^{\mu}} dy \right) \frac{F(x,h(w))}{|x|^\beta} ~dx\bigg|.
\end{align*}
So, now it is enough to prove that 
\begin{align*}
\int_{ \Om\cap\{h(w_k)\leq R\}} \left(\int_{{\mathbb R^N}}\frac{F(y,h(w_k))}{|y|^\beta|x-y|^{\mu}} dy \right) \frac{F(x,h(w_k))}{|x|^\beta} dx\to \int_{ \Om\cap\{h(w)\leq R\}} \left(\int_{{\mathbb R^N}}\frac{F(y,h(w))}{|y|^\beta|x-y|^{\mu}} dy \right) \frac{F(x,h(w))}{|x|^\beta}dx
\end{align*} as $k \to+\infty$.
Since $\left(\displaystyle\int_{\mathbb R^N}\frac{F(y,h(w))}{|y|^\beta|x-y|^\mu}dy\right)\frac{F(\cdot,h(w))}{|x|^\beta} \in L^{1}({\mathbb R^N})$, by Fubini's theorem, we  infer that
\begin{align*}
&\lim_{\Lambda \to+\infty} \int_{\Om \cap\{h(w)\leq R\}}\left(\int_{\{h(w)\geq \Lambda\}}\frac{F(y,h(w))}{|y|^\beta|x-y|^{\mu}}dy\right)\frac{F(x,h(w))}{|x|^\beta}~dx\\
&= \lim_{\Lambda \to+\infty} \int_{\{h(w)\geq \Lambda\}}\left(\int_{\Om \cap\{h(w)\leq R\}}\frac{F(y,h(w))}{|y|^\beta|x-y|^{\mu}}dy\right)\frac{F(x,h(w))}{|x|^\beta}~dx=0.
\end{align*}
Thus, we can fix  $\Lambda> \max\left\{\left(\frac{2C M_0}{\hat\de}\right)^{\frac{1}{m_0+1}}, s_0\right\}$, where $C$ is defined in \eqref{wk-sol10} and \eqref{wk-sol100},  such that
\[\int_{{\Om} \cap\{h(w)\leq R\}} \left(\int_{ \{h(w)\geq \Lambda\} }\frac{F(y,h(w) )}{|y|^\beta|x-y|^{\mu}} dy \right) \frac{F(x,h(w))}{|x|^\beta} ~dx \leq \hat \delta.\]
Moreover, from \eqref{wk-sol100}, $(f_4)$ and Lemma \ref{L1}-$(h_4)$, we deduce
\begin{align*}
&\int_{\Om \cap\{h(w_k)\leq R\}} \left(\int_{\{h(w_k)\geq \Lambda\}}\frac{F(y,h(w_k))}{|y|^\beta|x-y|^{\mu}} dy \right)\frac{F(x,h(w_k))}{|x|^\beta} ~dx\\
&\leq \frac{M_0}{{\Lambda}^{m_{0}+1}} \int_{\Om \cap\{h(w_k)\leq R\}} \left(\int_{\{h(w_k)\geq \Lambda\} }\frac{  f(y,h(w_k))h (w_k)(y)}{|y|^\beta|x-y|^{\mu}} dy \right) \frac{F(x,h(w_k))} {|x|^\beta}~dx\\
&\leq \frac{2M_0}{\Lambda^{m_{0}+1}} \int_{\mb R^N}\left(\int_{{\mathbb R^N} }\frac{F(y,h(w_k))}{|y|^\beta|x-y|^{\mu}} dy \right) \frac{f(x,h(w_k))} {|x|^\beta}h'(w_k)w_k ~dx\leq \hat\de.
\end{align*}
Therefore, using the last two relations, it follows that
\begin{align*}
&\left|\int_{\Om \cap\{h(w)\leq R\}} \left(\int_{\{h(w)\geq \Lambda\} }\frac{F(y,h(w) )}{|y|^\beta|x-y|^{\mu}} dy \right) \frac{F(x,h(w))}{|x|^\beta} ~dx\right.\\
&\quad \quad \left.- \int_{\Om \cap\{h(w_k)\leq R\}} \left(\int_{ \{h(w_k)\geq \Lambda\} }\frac{F(y,h(w_k))}{|y|^\beta|x-y|^{\mu}} dy \right) \frac{F(x,h(w_k))}{|x|^\beta} ~dx\right|\leq 2\hat\de.
\end{align*}
Now we claim that for fixed positive real numbers $R$ and $\Lambda$, the following holds:
\begin{equation}\label{choq-new}
\begin{split}
	&\ds\lim_{k\to+\infty}\left|\int_{\Om\cap\{h(w_k)\leq R\}} \left(\int_{\{h(w_k)\leq \Lambda\}}\frac{F(y,h(w_k))}{|y|^\beta|x-y|^{\mu}} dy \right) \frac{F(x,h(w_k))}{|x|^\beta} ~dx\right.\\
	&\qquad- \left.\int_{\Om\cap\{h(w)\leq R\}} \left(\int_{\{h(w)\leq \Lambda\}}\frac{F(y,h(w) )}{|y|^\beta|x-y|^{\mu}} dy \right) \frac{F(x,h(w))}{|x|^\beta} ~dx \right|= 0.
\end{split}
\end{equation}
It can easy be verified that  as $k \to+\infty$,
\begin{align}\label{a0}
&\left(\int_{\{h(w_k)\leq \Lambda\} }\frac{F(y,h(w_k))}{|y|^\beta|x-y|^{\mu}} dy \right) \frac{F(x,h(w_k))}{|x|^\beta}\chi_{ \Om\cap \{h(w_k)\leq R\}}(x)\notag\\  &\qquad\ra \left(\int_{\{h(w)\leq \Lambda\} }\frac{F(y,h(w))}{|y|^\beta|x-y|^{\mu}} dy \right) \frac{F(x,h(w))}{|x|^\beta}\chi_{ \Om\cap \{h(w)\leq R\}}(x)
\end{align}
point-wise a.e. 
Now taking $r=N$ in \eqref{k1}, using Lemma \ref{L1}-$(h_5)$ and \eqref{HLSineq}, we get a constant $C_{K,\Lambda}>0$ depending on $K$ and $\Lambda$ such that
\begin{align*}
	&\int_{\Om \cap \{h(w_k)\leq R\}}\left( \int_{\{h(w_k)\leq \Lambda\} }\frac{F(y,h(w_k))}{|y|^\beta|x-y|^{\mu}} dy \right) \frac{ F(x,h(w_k))}{|x|^\beta}dx  \notag\\
	&\leq  C_{R,\Lambda}\int_{\Om \cap \{h(w_k)\leq R\}}\left( \int_{\{h(w_k)\leq \Lambda\} }\frac{|h(w_k)(y)|^{N}}{|y|^\beta|x-y|^{\mu}} dy \right)  \frac{|h(w_k)(x)|^{N}}{|x|^\beta} dx \notag\\
	&\leq  C_{R,\Lambda}\int_{\Om\cap \{h(w_k)\leq R\}}\left( \int_{\{h(w_k)\leq \Lambda\} }\frac{|w_k(y)|^{N}}{|y|^\beta|x-y|^{\mu}} dy \right)  \frac{|w_k(x)|^{N}} {|x|^\beta}dx \notag\\
	& \leq C_{R,\Lambda} \int_{\Om}\left(\int_{{\mb R^N} }\frac{|w_k(y)|^{N}}{|y|^\beta|x-y|^{\mu}}~dy  \right) \frac{|w_k(x)|^{N}}{|x|^\beta} ~dx\notag\\
	& \leq { {C_{R,\Lambda}C(N,\mu,\beta)\|w_k\|_{L^{\frac{2N^2}{{2N-2\beta-\mu}}}(\Om)}^{2N} \to C_{R,\Lambda}C(N,\mu,\beta)\|w\|_{L^{\frac{2N^2}{{2N-2\beta-\mu}}}({\Om})}^{2N}}} \; \text{as}\; k \to+\infty.
\end{align*}  
Hence, by Theorem $4.9$ in \cite{bz} there exists some function $ \mc F \in L^1({\mathbb R^N})$ such that  up to a subsequence, still denoted by $\{w_k\}$, for each $k\in \mathbb N$, we have 
\[\left|\left( \int_{\{h(w_k)\leq \Lambda\} }\frac{F(y,h(w_k))}{|y|^\beta|x-y|^{\mu}} dy \right)  \frac{F(x,h(w_k))}{|x|^\beta}\chi_{ {\Om} \cap \{h(w_k)\leq R\}} \right| \leq | \mc F(x)|.\]
So, using \eqref{a0} and applying the Lebesgue dominated convergence theorem, we achieve \eqref{choq-new}.	This concludes the proof.
\end{proof}
\noi Next, we prove the following result regarding the regularity of  the convolution term with the singular weight, which we will use in the later context.
\begin{lemma}\label{reg}Let  \eqref{be} be satisfied and let the function $h$ be defined in \eqref{g}. Assume that $(f_1)$-$(f_3)$ hold. Then for any $w\in W^{1,N}(\mb R^N)$, we have
\begin{equation}\label{wk-sol7}
	\int_{\mb R^N} \frac{F(y,h(w))}{|y|^\beta|x-y|^{\mu}}dy\in L^\infty(\mb R^N).
\end{equation}
\end{lemma}
\begin{proof}
To prove this result,  we follow the idea as in \cite[Theorem 2.7]{gao-yang-du}.
For any $r>0$, we can write
\begin{equation}\label{UJM}
	\int_{\mb R^N} \frac{F(y,h(w))}{|y|^\beta|x-y|^{\mu}}dy\leq\int_{B_{r}(0)}\frac{F(y,h(w))}{|y|^\beta|x-y|^{\mu}}dy+\int_{\mathbb{R}^{N}\setminus B_{r}(0)}\frac{F(y,h(w))}{|y|^\beta|x-y|^{\mu}}dy.
\end{equation}
\noi Now for $x\in\mathbb{R}^{N}\setminus B_{2r}(0)$, we have $|x-y|>|y|$, which together with H\"older's inequality and the fact  $F(\cdot,h(w))\in L^{q}(B_r(0))$ for any $q>1$, implies that
\begin{equation}\nonumber
	\begin{aligned}
		\int_{B_{r}(0)}\frac{F(y,h(w))}{|y|^\beta|x-y|^{\mu}}dy
		<\int_{B_{r}(0)}\frac{F(y,h(w))}{|y|^{\mu+\beta}}dy\leq\|F(\cdot,h(w))\|_{L^{p^\prime}(B_{r}(0))}
		\int_{B_r(0)}\frac{1}{|y|^{(\mu+\beta)p}}dy<+\infty,
	\end{aligned}
\end{equation}
provided  $1<p<\frac{N}{\mu+\beta}$.\\
For $x\in B_{2r}(0)$, using the arguments as above, we deduce
\begin{equation}\nonumber
	\begin{aligned}
		\int_{B_{r}(0)}\frac{F(y,h(w))}{|y|^\beta|x-y|^{\mu}}dy
		&\leq\int_{B_{r}(0)}\frac{|F(y,h(w))|}{|y|^{\mu+\beta}}dy
		+\int_{B_{3r}(x)}\frac{F(y,h(w))}{|x-y|^{\mu+\beta}}dy<+\infty.
	\end{aligned}
\end{equation}
Therefore, for each $x\in \mb R^N$, we get
\begin{equation}\label{EDC}
	\int_{B_{r}(0)}\frac{F(y,h(w))}{|y|^\beta|x-y|^{\mu}}dy<+\infty.
\end{equation}
On the other hand, we have
$$
\int_{\mathbb{R}^{N}\setminus B_{r}(0)}\frac{F(y,h(w))}{|y|^\beta|x-y|^{\mu}}dy
=\int_{(\mathbb{R}^{N}\setminus B_{r}(0))\cap B_{r}(x)}\frac{F(y,h(w))}{|y|^\beta|x-y|^{\mu}}dy
+\int_{(\mathbb{R}^{N}\setminus B_{r}(0))\cap (\mathbb{R}^{N}-B_{r}(x))}\frac{F(y,h(w))}{|y|^\beta|x-y|^{\mu}}dy
$$
Again, following the above estimates, we obtain
\begin{equation}\nonumber
	\begin{aligned}
		\int_{(\mathbb{R}^{N}\setminus B_{r}(0))\cap B_{r}(x)}\frac{F(y,h(w))}{|y|^\beta|x-y|^{\mu}}dy
		&\leq\frac{1}{r^{\beta}}\int_{(\mathbb{R}^{N}\setminus B_{r}(0))\cap B_{r}(x)}\frac{F(y,h(w))}{|x-y|^{\mu}}dy\\
		&\leq\frac{1}{r^{\beta}}\int_{B_{r}(x)}\frac{F(y,h(w))}{|x-y|^{\mu}}dy<+\infty.\\
	\end{aligned}
\end{equation}
Similarly, using $F(\cdot,h(w))\in L^q (\mb R^N),$ for any $q\geq N,$ we have the below estimate
\begin{equation}\nonumber
	\begin{aligned}
		\int_{(\mathbb{R}^{N}\setminus B_{r}(0))\cap (\mathbb{R}^{N}\setminus B_{r}(x))}\frac{F(y,h(w))}{|y|^\beta|x-y|^{\mu}}dy
		&\leq\frac{1}{r^{\mu}}\int_{\mathbb{R}^{N}\setminus B_{r}(0)}\frac{F(y,h(w))}{|y|^{\beta}}dy\\
		&\leq\frac{1}{r^{\mu}}\|F(\cdot, h(w))\|_{L^{p}(\mathbb{R}^{N}\setminus B_{r}(0))}
		\int_{\mathbb{R}^{N}\setminus B_{r}(0)} \frac{1}{|y|^{\beta p}}dy<+\infty,
	\end{aligned}
\end{equation}
provided $p>\frac{N}{\beta}$.
Thus, 
\begin{equation}\label{YHN}
	\int_{\mathbb{R}^{N}\setminus B_{r}(0)}\frac{F(y,h(w))}{|y|^\beta|x-y|^{\mu}}dy<+\infty.
\end{equation}
Combining \eqref{UJM}, \eqref{EDC}, \eqref{YHN}, we can obtain our result.
\end{proof}
{\begin{lemma}\label{wk-sol}Assume that \eqref{be} holds. Let the assumptions $(f_1)$-$(f_6)$ be satisfied and the function $h$ be defined  in \eqref{g}. 	 If $\{w_k\}\subset W^{1,N}({\mathbb R^N})$ is a Cerami sequence for $ J$, then $\nabla w_k(x) \rightarrow \na w(x)$ a.e. in ${\mathbb R^N}$. Moreover, 
	\begin{equation}\label{wk-sol2}
		|\nabla w_k|^{N-2}\nabla w_k \rightharpoonup |\nabla w|^{N-2}\nabla w\; \text{weakly in}\; (L^{\frac{N}{N-1}}({\mathbb R^N}))^N \text{\;\; as $k\to+\infty$.}
	\end{equation}
\end{lemma}}
\begin{proof}
Since $\{w_k\}$  is a Cerami sequence for $J$, by Lemma \ref{lem712}, $\{w_k\}$  is bounded in $W^{1,N}({\mathbb R^N})$ and it satisfies \eqref{cnvv}.  
Thus, the sequence $\{|\nabla w_k|^{N-2}\nabla w_k\}$ is bounded in $(L^{\frac{N}{N-1}}({\mathbb R^N}))^N$. This implies that 
there exists $u \in (L^{\frac{N}{N-1}}({\mathbb R^N}))^N$ such that,
\begin{align*}|\nabla w_k|^{N-2}\nabla w_k \rightharpoonup u \; \text{weakly in}\;  (L^{\frac{N}{N-1}}({\mathbb R^N}))^N \; \text{as} \; k \to+\infty.\end{align*}
Also we have, $\{|\nabla w_k|^N+|w_k|^N\}$ is bounded in $L^1({\mathbb R^N})$,  which yields that there exists a non-negative radon measure $\sigma$ such that, up to a subsequence, we have
\begin{align*}|\nabla w_k|^N+| w_k|^N \to \sigma \; \text{in}\; (C({\Om}))^*\; \text{as}\; k \to+\infty,
\end{align*} for any $\Om\subset\subset \mb R^N.$
We show that $u = |\nabla w|^{N-2}\nabla w $.
For any fixed
$\nu>0$ we define the energy concentration set $$X_\nu := \left\{x \in  {\mathbb R^N}:\; \lim_{l\to 0}\lim_{k\to+\infty}\int_{B_l(x)} \left(|\nabla w_k|^N+|w_k|^N\right)dx\geq \nu\right\}.$$
Thus, $X_\nu$ is a finite set. 
Indeed, if not, then there exists a sequence of distinct points $\{z_k\}$ in $X_\nu$ such that, $\sigma(B_l(z_k))\geq \nu$ for all $l>0$ and for all $k\in\mathbb N$. This gives that $\sigma(\{z_k\}) \geq \nu$ for all $k$. Therefore, $\sigma(X_\nu)= +\infty$. But this is a contradiction to the fact that
\[\sigma(X_\nu) = \lim_{k \to+\infty} \int_{X_\nu} (|\nabla w_k|^N+|w_k|^N )~dx \leq C.\] 
Thus, we can take $X_\nu= \{z_1,z_2,\ldots, z_n\}$.\\
Next we claim that, we can choose sufficiently small $\nu>0,$  with $\nu^{\frac{1}{N-1}} < \frac{1}{2^{\frac{1}{N-1}}} \frac{{2N-2\beta-\mu}}{2N}\left(1-\frac{\beta}{N}\right)\frac {\alpha_N}{\al_0}$ such that
\begin{align}\label{wk-sol6}
	\lim_{k \to+\infty} \int_\Om\left( \int_{\mathbb R^N} \frac{F(y, h (w_k))}{|y|^\beta|x-y|^{\mu}}dy\right) \frac{f(x,h(w_k))}{|x|^\beta}h'(w_k)w_k~dx=  \int_\Om \left( \int_{\mathbb R^N} \frac{F(y,h(w))}{|y|^\beta|x-y|^{\mu}}dy\right) \frac{f(x,h(w))}{|x|^\beta} h'(w)w~dx, \end{align} where  $\Om$ is any compact subset of $ {\mathbb R^N} \setminus X_\nu$.\\
Let us choose $z_0 \in \Om$ and $l_0>0$ be such that $\sigma(B_{l_0}(z_0)) < \nu$. Hence  $z_0 \notin X_\nu$. Furthermore, we consider  $\phi \in C_c^\infty({\mathbb R^N})$ with $0\leq \phi(x)\leq 1$ for $x \in {\mathbb R^N}$, $\phi \equiv 1$ in $B_{\frac{l_0}{2}}(z_0)$ and $\phi \equiv 0$ in $ {\mathbb R^N} \setminus B_{l_0}(z_0)$. Then
\[\lim_{k \to+\infty} \int_{B_{\frac{l_0}{2}}(z_0) }(|\nabla w_k|^N+|w_k|^N) dx\leq \lim_{k \to+\infty} \int_{B_{l_0}(z_0) }(|\nabla w_k|^N+|w_k|^N)\phi dx \leq \sigma(B_{l_0}(z_0) ) < \nu. \]
Therefore, for sufficiently large $k \in \mb N$ and  sufficiently small $\e>0$, we have
\begin{equation}\label{wk-sol3}
	\int_{B_{\frac{l_0}{2}}(z_0)}(|\nabla w_k|^N+|w_k|^N) dx\leq \nu(1-\e).
\end{equation}
From $(f_2)$, $(f_3)$,  for  $\al>\al_0$, very close to $\al_0$, there exists some constant $C>0$ such that $|f(x,s)|\leq C\exp\left( \al |s|^{\frac{2N}{N-1}}\right)$ in $B_{\frac{l_0}{2}}(z_0)$. Using this, together with 
\eqref{wk-sol3} and  Lemma \ref{L1}-$(h_6
)$, we deduce
{\small	\begin{equation*}
		\begin{split}
			 \int_{B_{\frac{l_0}{2}}(z_0)}\frac{|f(x,h(w_k))|^q}{|x|^\beta}~dx 
			&  \leq C\int_{B_{\frac{l_0}{2}}(z_0)}\frac{\exp\left(\al q 2^{\frac{1}{N-1}}|w_k|^{\frac{N}{N-1}}\right)}{|x|^\beta}~dx\\
			& \leq C \int_{B_{\frac{l_0}{2}}(z_0)}\frac{1}{|x|^\beta}\exp\Big( \al q 2^{\frac{1}{N-1}}\nu^{\frac{1}{N-1}}(1-\e)^{\frac{1}{N-1}}\Bigg(\frac{|w_k|^N}{\int_{B_{\frac{l_0}{2}}(z_0) }(|\nabla w_k|^N+| w_k|^N) dx}\Bigg)^{\frac{1}{N-1}}\Big)~dx
		\end{split}
\end{equation*}}
\noi Hence, we can choose $q>1$  such that $ \al q 2^{\frac{1}{N-1}}\nu^{\frac{1}{N-1}}(1-\e)^{\frac{1}{N-1} }< \left(1-\frac{\beta}{N}\right)\frac{\alpha_N}{\al_0}$ and using Theorem \ref{as} in the last relation, we get
\begin{align}\label{aa}
	& \int_{B_{\frac{l_0}{2}}(z_0)}\frac{|f(x,h(w_k))|^q}{|x|^\beta}~dx\leq C.
\end{align} Next, we consider
\begin{equation*}
	\begin{split}
		&\int_{B_{\frac{l_0}{2}}(z_0) } \left|\left( \int_{\mathbb R^N} \frac{F(y,h(w_k))}{|y|^\beta |x-y|^{\mu}}dy\right) \frac{f(x,h(w_k))}{|x|^\beta}h'(w_k) w_k- \left( \int_{\mathbb R^N} \frac{F(y,h(w))}{|y|^\beta|x-y|^{\mu}}dy\right) \frac{f(x,h(w))}{|x|^\beta}h'(w) w \right|~dx\\
		& \leq \int_{B_{\frac{l_0}{2}}(z_0)} \left( \int_{\mathbb R^N} \frac{F(y,h(w))}{|y|^\beta|x-y|^{\mu}}dy\right) \left|\frac{f(x,h(w_k))}{|x|^\beta}h'(w_k) w_k-\frac{f(x,h(w))}{|x|^\beta}h'(w))w\right|~dx\\
		& \quad + \int_{B_{\frac{l_0}{2}}(z_0) } \left|\left( \int_{\mathbb R^N} \frac{F(y,h(w_k))-F(y,h(w))}{|y|^\beta|x-y|^{\mu}}dy\right) \frac{f(x,h(w_k))}{|x|^\beta}h'(w_k) w_k\right|~dx\\
		&:= L_1 + L_2.
	\end{split}
\end{equation*}
\noi Now	using \eqref{wk-sol7}, we get the estimate
\begin{align}\label{wk-sol8}
	L_1 \leq  C\int_{B_{\frac{l_0}{2}}(z_0)}  \left |\frac{f(x,h(w_k))}{|x|^\beta}h'(w_k)w_k-\frac{f(x,h(w))}{|x|^\beta}h'(w)w \right |~dx,
\end{align}
where $C>0$ is some constant. 
Moreover, 	the asymptotic growth assumptions on  $f$ gives us
\begin{equation*}
	\lim_{s \to+\infty} \frac{f(x,s
		)s}{(f(x,s))^r} = 0\; \text{uniformly in }x \in {\mathbb R^N},\; \text{for all}\; r>1.
\end{equation*} This  together with \eqref{aa} and Lemma \ref{L1}-$(h_4)$  implies that $$\int_{B_{\frac{l_0}{2}}(z_0)}\ds\frac {f(x,h(w_k))}{|x|^\beta}h'(w_k)w_kdx\leq\int_{B_{\frac{l_0}{2}}(z_0)}\ds\frac {f(x,h(w_k))}{|x|^\beta}h(w_k)dx\leq\int_{B_{\frac{l_0}{2}}(z_0)}\ds\frac {(f(x,h(w_k)))^q}{|x|^\beta}dx\leq C.$$  That is $\left\{\ds\frac {f(x,h(w_k))}{|x|^\beta}h'(w_k)w_k\right\}$ is an equi-integrable family of functions over ${B_{\frac{l_0}{2}}(z_0)} $. Also, by the continuity of $f $ and $h$ we get,   $\frac {f(x,h(w_k))}{|x|^\beta}h'(w_k)w_k(x)$ $\to$ $\frac {f(x,h(w))}{|x|^\beta}h'(w)w(x)$  a.e. in $\mb R^N$, as $k \to+\infty.$ Hence,  we obtain $L_1 \to 0$ as $k\to+\infty$, thanks to Vitali's convergence theorem. Next, we show that $L_2 \to 0$ as $k \to+\infty$.\\
For that, first we get the following estimate  with the help of the semigroup property of the Riesz Potential and the Cauchy-Schwartz inequality:
\begin{align*}
	\begin{split}
		&\int_{{\mathbb R^N}} \left(\int_{{\mathbb R^N}}\frac{F(y,h(w_k))-F(y,h(w))}{|y|^\beta|x-y|^{\mu}} dy \right) \chi_{B_{\frac{l_0}{2}}(z_0) }(x) \frac{f(x,h(w_k))}{|x|^\beta}h'(w_k) w_k ~dx \\
		&\leq  C\left(\int_{B_{\frac{l_0}{2}} (z_0)} \left( \int_{{\mathbb R^N}}\frac{|F(y,h(w_k))- F(y,h(w))| dy }{|y|^\beta|x-y|^{\mu}}\right) \frac{|F(x,h(u_k))- F(x,h(u))|}{|x|^\beta} ~dx \right)^{\frac12}\\
		&\quad \times \left(\int_{{\mathbb R^N}} \left(\int_{{\mathbb R^N}} \chi_{B_{\frac{l_0}{2}}(z_0)}(y) \frac{f(y,h(w_k))h'(w_k) w_k}{|y|^\beta|x-y|^{\mu}} dy \right) \chi_{B_{\frac{l_0}{2}} (z_0)}(x) \frac{f(x,h(w_k))}{|x|^\beta}h'(w_k) w_k ~dx \right)^{\frac12},
	\end{split}
\end{align*}where $C$ is some constant, independent of $k$.
By Theorem \ref{HLS} with $\beta=\vartheta$ and	gathering \eqref{aa}, \eqref{wk-sol8}, and choosing $\nu$  sufficiently small such that  $\nu^{\frac{1}{N-1}} < \frac{1}{2^{\frac{1}{N-1}}}\frac{{2N-2\beta-\mu}}{2N} \left(1-\frac{\beta}{N}\right) \frac{\alpha_N}{\al_0}$, we find that
\begin{align}\label{bdf}
	&\left(\int_{{\mathbb R^N}} \left(\int_{{\mathbb R^N}} \chi_{B_{\frac{l_0}{2}} (z_0)}(y) \frac{f(y,h(w_k))}{|y|^\beta|x-y|^\mu}h'(w_k) w_k dy \right) \chi_{B_{\frac{l_0}{2}}(z_0) }(x) \frac{f(x,h(w_k))}{|x|^\beta}h'(w_k) w_k ~dx \right)^{\frac12}\notag\\& \leq \|\chi_{B_{\frac{l_0}{2}} (z_0)}f(\cdot,h(w_k))h'(w_k)w_k \|_{L^{\frac{2N}{{2N-2\beta-\mu}}}({\mathbb R^N})} \leq C.
\end{align}
Now  we claim that \begin{align}\label{3.233}
	\lim_{k \to+\infty} \int_{B_{\frac{l_0}{2}} (z_0)} \left(\int_{{\mathbb R^N}}\frac{|F(y,h(w_k))-F(y,h(w))|}{|y|^\beta|x-y|^{\mu}} dy \right) \frac{|F(x,h(w_k))-F(x,h(w))|} {|x|^\beta} ~dx =0.
\end{align}  Using the similar arguments used in Lemma \ref{PS-ws}, we get 
as $\Lambda \to+\infty$
\begin{align}\label{3.25}
	&\int_{B_{\frac{l_0}{2}} (z_0)} \int_{|h(w)| \geq \Lambda\}} \frac{F(y,h(w))}{|y|^\beta|x-y|^{\mu}} \frac{F(x,h(w))}{|x|^\beta} dy ~dx=o(\Lambda),\\ &\int_{B_{\frac{l_0}{2}} (z_0)} \int_{\{|h(w_k)| \geq \Lambda\}} \frac{F(y,h(w_k))}{|y|^\beta|x-y|^{\mu}} \frac{F(x,h(w_k))}{|x|^\beta} dy ~dx = o(\Lambda)\label{3.20},
\end{align}
\begin{equation}\label{3.26}
	\int_{B_{\frac{l_0}{2}} (z_0)} \int_{\{|h(w)| \geq \Lambda\}} \frac{F(y,h(w_k))}{|y|^\beta|x-y|^{\mu}} \frac{F(x,h(w))} {|x|^\beta}dy ~dx = o(\Lambda),
\end{equation}
and
\begin{equation}\label{3.27}
	\int_{B_{\frac{l_0}{2}} (z_0)} \int_{\{|h(w_k)| \geq \Lambda\}} \frac{F(y,h(w_k))}{|y|^\beta|x-y|^{\mu}}\frac{ F(x,h(w))}{|x|^\beta} dy ~dx = o(\Lambda).
\end{equation}
So,
\begin{equation*}
	\begin{split}
		& \int_{B_{\frac{l_0}{2}} (z_0)} \left(\int_{{\mathbb R^N}}\frac{|F(y,h(w_k))-F(y,h(w))|}{|y|^\beta|x-y|^{\mu}} dy \right) \frac{|F(x,h(w_k))-F(x,h(w))|}{|x|^\beta}  ~dx\\&\leq
		2\int_{B_{\frac{l_0}{2}} (z_0)} \left(\int_{{\mathbb R^N}}\frac{\chi_{\{h(w_k)\geq \Lambda\}}(y)F(y,h(w_k))}{|y|^\beta|x-y|^{\mu}} dy \right)\frac{ F(x,h(w_k))}{|x|^\beta} ~dx \\
		&+4 \int_{B_{\frac{l_0}{2}} (z_0)} \left(\int_{{\mathbb R^N}}\frac{F(y,h(w_k))}{|y|^\beta|x-y|^{\mu}} dy \right) \chi_{\{h(w)\geq \Lambda\}}(x)\frac{F(x,h(w))}{|x|^\beta}~dx\\&+4 \int_{B_{\frac{l_0}{2}} (z_0)} \left(\int_{{\mathbb R^N}}\frac{\chi_{\{h(w_k)\geq \Lambda\}}(y)F(y,h(w_k))}{|y|^\beta|x-y|^{\mu}} dy \right)\frac{ F(x,h(w))} {|x|^\beta}~dx\\
		&+2\int_{B_{\frac{l_0}{2}} (z_0)} \left(\int_{{\mathbb R^N}}\frac{\chi_{\{h(w)\geq \Lambda\}}(y)F(y,h(w))}{|y|^\beta|x-y|^{\mu}} dy \right) \frac{F(x,h(w))}{|x|^\beta} ~dx \\
		&+\int_{B_{\frac{l_0}{2}} (z_0)}\Bigg[\left(\int_{\mathbb R^N}\frac{|F(y,h(w_k))\chi_{\{h(w_k)\leq \Lambda\}}(y)-F(y,h(w))\chi_{\{h(w)\leq \Lambda\}}(y)|}{|y|^\beta|x-y|^\mu}dy\right)\\&\qquad\qquad\qquad\frac{|F(x,h(w_k))\chi_{\{h(w_k)\leq \La\}}(x)-F(x,h(w))\chi_{\{h(w)\leq \Lambda\}}(x)|}{|x|^\beta}\Bigg]~dx.
	\end{split}
\end{equation*}
Then the Lebesgue dominated convergence theorem yields that the last integration tends to $0$ as $k \to+\infty.$ Recalling \eqref{3.25}-\eqref{3.27}, we get \eqref{3.233}, which together with \eqref{bdf}	implies that $L_2 \to 0$ as $k \to+\infty$.
This yields that
\begin{align*}&\lim_{k \to+\infty}\int_{B_{\frac{l_0}{2}}(z_0)} \Bigg|\left( \int_{\mathbb R^N} \frac{F(y,h(w_k))}{|y|^\beta|x-y|^{\mu}}dy\right) \frac{f(x,h(w_k))}{|x|^\beta}h'(w_k)w_k\\&\qquad\qquad\qquad\qquad\qquad- \left( \int_{\mathbb R^N} \frac{F(y,h(w))}{|y|^\beta|x-y|^{\mu}}dy\right) \frac{f(x,h(w))}{|x|^\beta}h'(w)w \Bigg|~dx=0.\end{align*}
Since $\Om$ is compact, by applying standard finite covering lemma, we obtain \eqref{wk-sol6}. Rest of the proof of \eqref{wk-sol2} can be concluded by following the classical arguments as in the proof of Lemma $4$ in \cite{M-do1} and Lemma 4.4 in \cite{ay}. 
\end{proof}
\begin{lemma}\label{VV}
Assume that \eqref{be} is satisfied.	Suppose that $(V_1)$, $(V_2^\prime$) and $(f_1)$-$(f_6)$   hold and  $h$ is defined as in \eqref{g}. 	 Let $\{w_k\}\subset W^{1,N}({\mathbb R^N})$ be a Cerami sequence for $ J$. Then
\begin{align*} 
\int_{\mb R^N}	V(x)|h(w_k)|^{N-2}h(w_k) h'(w_k)\phi dx \to \int_{\mb R^N} V(x)|h(w)|^{N-2}h(w) h'(w)\phi dx \text{\; for any \; } \phi \in C_c^\infty(\mb R^N) \text{\; as $k\to+\infty$.}
\end{align*} 
\end{lemma} 
\begin{proof}
Let $\varphi\in C_c^\infty(\mb R^N)$ such that $supp (\varphi) =\Om$. Then using   Lemma \ref{L1}-$(h_3)$,$(h_5)$, we get
\begin{align*} 
	\left| V(x)|h(w_k)|^{N-2}h(w_k) h'(w_k)\varphi  \right |
\leq \|V\|_{L^\infty(\Om) }|w_k|^{N-1}|\varphi|.
\end{align*}
Now   H\"older's inequality and the continuous embedding $E\hookrightarrow L^N(\mb R^N)$ imply  
\[\int_\Om  \|V\|_{L^\infty(\Om)} |w_k|^{N-1}|\varphi| dx\leq \|V\|_{L^\infty(\Om)} \|w_k\|_{L^N(\Om)}^{N-1}\|\varphi\|_{L^N(\Om)} \leq  \|V\|_{L^\infty(\Om)} \|w_k\|^{N-1}\|\varphi\|_{L^N(\Om)}\leq C, \] 
for some positive constant $C$, independent in $k$, where in the last line we used the fact that $\{w_k\}$ is bounded in $W^{1,N}(\mb R^N)$ by Lemma \ref{lem712}.  Finally by  applying the Lebesgue dominated convergence theorem, we conclude the proof.
\end{proof}
{\begin{lemma}\label{kc-ws} Assume that \eqref{be} and $(f_1)$-$(f_5)$ hold and the function  $h$ is defined as  in \eqref{g}. Let $\{w_k\}\subset W^{1,N}({\mathbb R^N})$ be a Cerami sequence for $J$. Then  for all $\psi\in C_c^\infty(\mb R^N)$, as $k\to+\infty$, we have
	\begin{align*}
		\int_{\mathbb R^N} \left(\int_{\mathbb R^N} \frac{F(y, h(w_k))}{|y|^\beta|x-y|^{\mu}}dy\right)\frac{f(x, h(w_k))}{|x|^\beta}h'(w_k)\psi ~dx\to \int_{\mathbb R^N} \left(\int_{\mathbb R^N} \frac{F(y,h(w))}{|x-y|^{\mu}}dy\right)\frac{f(x,h(w))}{|x|^\beta}h'(w)\psi~dx. \end{align*}
\end{lemma}}
\begin{proof}
Let $\Om$ be any compact subset of $\mb R^N$. Let $\psi \in C_c^\infty({\mathbb R^N})$ with compact support $\Om' (\supset \Om)$ such that  $\psi \equiv 1$ in $\Om $ and $0\leq \psi \leq 1$ in $\Om'$. It can be easily computed that
\begin{equation}\label{k}
	\begin{split}
		\left\| \frac{\psi}{1+w_k}\right\|^N &= \int_{\mathbb R^N} \left|\frac{\nabla \psi}{1+w_k}- \psi \frac{\nabla w_k}{(1+w_k)^2} \right|^N~dx +\int_{\mathbb R^N}  \frac{|\psi|^N}{|1+w_k|^N} dx\nonumber \\
		&\leq 2^{N-1}(\|\psi\|^N+ \|w_k\|^N),
	\end{split}
\end{equation}
that is $\frac{\psi}{1+w_k} \in W^{1,N}({\mathbb R^N})$. Now in \eqref{kc-PS-bdd1}, choosing $\phi:=\frac{\psi}{1+w_k}$   as a test function and using Lemma \ref{L1}-$(h_3)$, $(h_5)$, H\"older's inequality and \eqref{k}, we obtain
\begin{equation}\label{kc-ws-new1}
	\begin{split}
		&\int_{{\Om}}\left( \int_{\mathbb R^N} \frac{F(y,h(w_k))}{|y|^\beta|x-y|^\mu}dy\right)\frac{f(x,h(w_k))}{|x|^\beta(1+w_k)}h'(w_k)~dx\\& \leq \int_{\mathbb R^N} \left( \int_{\mathbb R^N} \frac{F(y,h(w_k))}{|y|^\beta|x-y|^\mu}dy\right)\frac{f(x,h(w_k)) h'(w_k)\psi}{|x|^\beta(1+w_k)}~dx\\
		&=  \e_k \left\|\frac{\psi}{1+w_k}\right\| + \int_{\mathbb R^N}  |\nabla w_k|^{N-2}\nabla w_k \nabla \left( \frac{\psi}{1+w_k}\right)~dx+\int_{\mathbb R^N}  V(x)| h (w_k)|^{N-2} h (w_k) h'(w_k) \left( \frac{\psi}{1+w_k}\right)~dx\\
		& \leq \e_k 2^{\frac{N-1}{N}}(\|\psi\|+ \|w_k\|) +\int_{\mathbb R^N} |\nabla w_k|^{N-2}\nabla w_k \left(\frac{\nabla \psi}{1+w_k}-\psi\frac{\nabla w_k}{(1+w_k)^2}\right)~dx+\int_{\mathbb R^N}  V(x)|  w_k|^{N-2}  w_k  \left( \frac{\psi}{1+w_k}\right)~dx\\
		& \leq \e_k 2^{\frac{N-1}{N}}(\|\psi\|+ \|w_k\|) +\int_{\mathbb R^N} |\nabla w_k|^{N-1} \left( |\nabla \psi|+ |\nabla w_k|\right)~dx+ \|V\|_{L^\infty(\Om')}\int_{\Om'}|w_k|^{N-1}dx \\
		& \leq \e_k 2^{\frac{N-1}{N}}(\| \psi\|+ \| w_k\|)+ [\|\nabla \psi\|_{L^N(\mb R^N)}\|\nabla w_k\|_{L^N(\mb R^N)}^{N-1}+ \|\nabla w_k\|_{L^N(\mb R^N)}^N]+\|w_k\|_{L^{N-1}(\Om')}^{N-1}\leq C_1,
	\end{split}
\end{equation}
where $C_1:=C_1(\Om)$ is a positive constant and  in the last line we used the facts that sequence $\{w_k\}$ is bounded in $W^{1,N}({\mathbb R^N})$ by Lemma \ref{lem712} and $w_k\to w$ strongly in $L^{N-1}(\Om')$.
Gathering \eqref{kc-ws-new1} and \eqref{wk-sol100}, we deduce
\begin{align*}
	&\int_{{\Om}}\left( \int_{\mathbb R^N} \frac{F(y,h(w_k))}{|y|^\beta|x-y|^\mu}dy\right)\frac{f(x,h(w_k))}{|x|^\beta}h'(w_k)~dx \\
	& \leq  2\int_{{\Om}\cap \{w_k <1\}} \left( \int_{\mathbb R^N}\frac{F(y,h(w_k))}{|y|^\beta|x-y|^\mu}dy\right)\frac{f(x,h(w_k))h'(w_k)}{|x|^\beta(1+w_k)}~dx\\& + \int_{{\Om}\cap \{w_k \geq 1\}} \left(\int_{\mathbb R^N} \frac{F(y,h(w_k))}{|x-y|^\mu}dy\right)\frac{f(x,h(w_k))}{|x|^\beta}w_k h'(w_k)~dx\\
	& \leq 2\int_{{\Om}} \left( \int_{\mathbb R^N}\frac{F(y,h(w_k))}{|y|^\beta|x-y|^\mu}dy\right)\frac{f(x,h(w_k))h'(w_k)}{|x|^\beta(1+w_k)}~dx\\& + \int_{\mb R^N}\left( \int_{\mathbb R^N} \frac{F(y,h(w_k))}{|y|^\beta|x-y|^\mu}dy\right)\frac{f(x,h(w_k))}{|x|^\beta}h'(w_k) w_k~dx\\
	& \leq 2C_1+C  :=C_2.
\end{align*}
Thus, the sequence $\{v_k\}:=\left\{\left( \int_{\mathbb R^N}\frac{F(y,h(w_k))}{|y|^\beta|x-y|^\mu}dy\right){\frac{f(x,h(w_k))}{|x|^\beta}h'(w_k)}\right\}$ is bounded in $L^1_{\text{loc}}({\mathbb R^N})$. Hence, there is a radon measure $\zeta$ such that, up to a subsequence, $w_k \rightharpoonup \zeta$ in the ${weak}^*$-topology as $k \to+\infty$.  Therefore,  we have
\[\lim_{k \to+\infty}\int_{\Om}v_k\varphi=\lim_{k \to+\infty}\int_{\Om}\int_{\mathbb R^N} \left( \frac{F(y,h(w_k))}{|y|^\beta|x-y|^\mu}dy\right)\frac{f(x,h(w_k))}{|x|^\beta}h'(w_k)\varphi ~dx = \int_{\mathbb R^N} \varphi ~d\zeta,\; \forall \varphi \in C_c^\infty({\Om}). \]
Since $w_k$ satisfies \eqref{kc-PS-bdd1}, we achieve
\[\int_{\Om} \varphi d\zeta= \lim_{k \to+\infty} \int_{\mb R^N} |\nabla w_k|^{N-2}\nabla w_k \nabla \varphi ~dx, \;\;\text { for all } \varphi\in C_c^\infty( {\Om}),  \]
which together with Lemma~\ref{wk-sol} yields that the Radon measure $\zeta$ is absolutely continuous with respect to the Lebesgue measure. So, by Radon-Nikodym theorem, since $\Om\subset \mb R^N$ is an arbitrary compact set, there exists a function $\varrho \in L^1_{\text{loc}}({\mathbb R^N})$ such that for any $\psi\in C^\infty_c({\mathbb R^N})$, we have $\int_{\mathbb R^N} \psi~ d\zeta= \int_{\mathbb R^N} \psi \varrho~dx$.
Therefore, we get
\begin{align*}&\lim_{k \to+\infty}\int_{\mathbb R^N}\left( \int_{\mathbb R^N} \frac{F(y,h(w_k))}{|y|^\beta|x-y|^\mu}dy\right)\frac{f(x,h(w_k))}{|x|^\beta}h'(w_k)\psi(x)~ ~dx\\&\qquad = \int_{\mathbb R^N} \psi \varrho~dx= \int_{\mathbb R^N}  \left( \int_{\mathbb R^N} \frac{F(y,h(w))}{|y|^\beta|x-y|^\mu}dy\right)\frac{f(x,h(w))}{|x|^\beta}h'(w)\psi(x)~ ~dx, \;\;\text{ for all } \psi\in C^\infty_c({\mathbb R^N}). \end{align*}
This completes the proof.
\end{proof}

\section{Mountain pass geometry}
\noi	In  this section first we study the mountain pass structure to the energy functional  $J: W^{1,N}({\mathbb R^N})\rightarrow \mb R$ associated to the problem \eqref{pp}.
Now we obtain the following  necessary result  to prove the next lemma.
{	\begin{lemma}\label{clm-mp}
	Let the function $h$ be defined as in \eqref{g}. Let $V$ satisfy $(V_1)$, $(V_2^\prime)$. Then	for any $w\in W^{1,N}(\mb R^N)$,	 there exist constants $C, \la_*>0$ such that 
	\begin{align}\label{09}
		\int_{\mb R^N} |\nabla w|^N dx+	\int_{\mb R^N} V(x) |h(w)|^N dx\geq C\|w\|^N,\text{\;\; whenever \;} \|w\|<\la_*.
	\end{align}
\end{lemma}}
\begin{proof}
Suppose \eqref{09} does not hold. Then for each $k\in \mathbb N,$ there exists $w_k(\not= 0)\in W^{1,N}(\mb R^N),$ such that $w_k\to 0$ strongly in $W^{1,N}(\mb R^N)$
as $k\to+\infty$ and $$\int_{\mb R^N} |\nabla w_k|^N dx+	\int_{\mb R^N} V(x) |h(w_k)|^N dx< \frac 1k\|w_k\|^N.$$ Let us set $v_k:=\frac{w_k}{\|w_k\|}.$
Using this  in the last relation, we deduce \begin{align}\label{mp1.0}
	\int_{\mb R^N} |\nabla v_k|^N dx+	\int_{\mb R^N} V(x) |v_k|^N dx+\int_{\mb R^N} V(x)\left(\frac{|h(w_k)|^N}{|w_k|^N}-1\right) |v_k|^N dx< \frac{1}{k}.
\end{align} Since $w_k\to 0$ in $W^{1,N}(\mb R^N)$, we have $w_k\to 0$ strongly in $L_{loc}^q(\mb R^N)$ for $1\leq q<\infty$ and $w_k(x)\to 0$ a.e. in $\mb R^N.$ Therefore, for any given $\e>0,$ as $k\to+\infty,$ the measure
$$\left| \left\{ x\in\mb R^N\;:\;|w_k|>\e\right\}\right| \to 0.$$ Hence, for $q>N$, using H\"older's inequality and continuous embedding $ W^{1,N}(\mb R^N)\hookrightarrow L^q(\mb R^N)$  along with  the last limit, we obtain 
\begin{align}\label{mp1.1}
	\int_{\left\{ x\in\mb R^N\;:\;|w_k|>\e\right\}}  |v_k|^N dx\leq C\left| \left\{ x\in\mb R^N\;:\;|w_k|>\e\right\}\right|^{\frac{q-N}{q}}\|v_k\|^N\to 0.
\end{align} 
Now using \eqref{mp1.1} and Lemma \ref{L1}-$(h_{10})$,  as $k\to+\infty$, applying Vitali's convergence theorem, we get $$\int_{\mb R^N} V(x)\left(\frac{|h(w_k)|^N}{|w_k|^N}-1\right) |v_k|^N dx\to0.$$ Therefore, from \eqref{mp1.0} and $(V_1)$ we get $\|v_k\|\to0$ which contradicts the fact that $\|v_k\|=1.$ Thus, the lemma is proved.
\end{proof}
		\begin{lemma}\label{lem7.1}
			Let \eqref{be} hold and the function $h$ be defined in \eqref{g}.	Suppose that  the conditions $(V_1)$, $(V_2^\prime)$ and $(f_1)$-$(f_5)$ hold.  Then  there exist $\la_\beta>0$ and $\tau_\beta>0$ such that
			\begin{align*}
				J(w)\geq \tau_\beta>0 \quad \text{for any } w\in W^{1,N}({\mathbb R^N}) \text{ with }\ \|w\|=\la_\beta.
			\end{align*}
		\end{lemma}
		\begin{proof} Let $w\in W^{1,N}({\mathbb R^N})$.	Then using \eqref{k2},  \eqref{k1}, Lemma \ref{L1}-$(h_5)$, Lemma \ref{techn},  Lemma \ref{mp-inq},  H\"{o}lder's inequality, and the Sobolev embedding,    we deduce
			\begin{align}\label{kc-MP1}
				&\int_{{\mathbb R^N}}\left(\int_{{\mathbb R^N}} \frac{F(y,h(w))}{|y|^\beta|x-y|^{\mu}}dy\right)\frac{F(x,h(w))}{|x|^\beta}~dx \notag\\
				& \leq C(N,\mu,\beta)\left\|\e |h(w)|^N + C |h(w)|^r \left[\exp\left(\al|h(w)|^{\frac{2N}{N-1}}\right)-S_{N-2}(\al, (h(w))^2)\right]\right\|_{L^{\frac{2N}{{2N-2\beta-\mu}}}(\mb R^N)}^2\notag\\
				&\leq C(N,\mu,\beta )\bigg[2^{\frac{2N}{{2N-2\beta-\mu}}}\bigg\{\e^{\frac{2N}{{2N-2\beta-\mu}}} \int_{{\mathbb R^N}} |h(w)|^{\frac{2N^2}{{2N-2\beta-\mu}}} dx +  C^{\frac{2N}{{2N-2\beta-\mu}}} \int_{{\mathbb R^N}} |h(w)|^{\frac{2Nr}{{2N-2\beta-\mu}}}\notag\\&\qquad\qquad\qquad\left(\exp\left(\frac{2N}{{2N-2\beta-\mu}}\al|h(w)|^{\frac{2N}{N-1}}\right)-S_{N-2}\left(\frac{2N}{{2N-2\beta-\mu}}\al,|h(w)|^2\right)  \right)dx\bigg\} \bigg]^{\frac{{2N-2\beta-\mu}}{N}} \notag\\
				&\leq 4C(N,\mu,\beta)\Bigg[\e^{2}\left( \int_{{\mathbb R^N}} |w|^{\frac{2N^2}{{2N-2\beta-\mu}}} dx \right)^{\frac{{2N-2\beta-\mu}}{N}} \notag\\ & \quad \qquad\qquad\qquad+  C^{2} 2^{\frac rN} \bigg\{\int_{{\mathbb R^N}} |w|^{\frac{Nr}{{2N-2\beta-\mu}}}\bigg(\exp\left(\frac{2N \al 2^{\frac{1}{N-1}}}{{2N-2\beta-\mu}}|w|^{\frac{N}{N-1}}\right)\notag\\&\qquad\qquad\qquad\qquad\qquad\qquad\qquad\qquad\qquad-S_{N-2}\left(\frac{2N\al 2^{\frac{1}{N-1}}}{{2N-2\beta-\mu}},|w|\right)  dx\bigg)\bigg\}^{\frac{{2N-2\beta-\mu}}{N}} \Bigg] \notag\\
				&\leq C_1(N,\mu,\beta,\e) \Bigg[ \|w\|^{2N}_{L^{\frac{2N^2}{{2N-2\beta-\mu}}}({\mathbb R^N})} + \|w\|^{2r}_{L^{\frac{2Nr }{{2N-2\beta-\mu}}}({\mathbb R^N})} \Bigg\{\int_{{\mathbb R^N}}\Bigg(\exp\left(\frac{4N\al2^{\frac{1}{N-1}}}{{2N-2\beta-\mu}}{|w|}^{\frac{N}{N-1}}\right)\notag\\&\qquad\qquad\qquad\qquad\qquad\qquad\qquad\qquad\qquad\qquad-S_{N-2}\left(\frac{4N\al 2^{\frac{1}{N-1}}}{{2N-2\beta-\mu}},|w|\right)\Bigg)dx\Bigg\}^{\frac{{2N-2\beta-\mu}}{2N}} \Bigg]\notag\\
				&\leq C_2(N,\mu,\beta,\e) \Bigg[ \|w\|^{2N} + \|w\|^{2r} \Bigg\{\int_{{\mathbb R^N}}\Bigg(\exp\left(\frac{4N\al2^{\frac{1}{N-1}}}{{2N-2\beta-\mu}}{|w|}^{\frac{N}{N-1}}\right)\notag\\&\qquad\qquad\qquad\qquad\qquad\qquad\qquad\qquad\qquad\qquad-S_{N-2}\left(\frac{4N\al 2^{\frac{1}{N-1}}}{{2N-2\beta-\mu}},|w|\right)\Bigg)dx\Bigg\}^{\frac{{2N-2\beta-\mu}}{2N}} \Bigg]\notag\\
				&\leq C_2(N,\mu,\beta,\e) \Bigg[ \|w\|^{2N} + \|w\|^{2r} \Bigg\{\int_{{\mathbb R^N}}\Bigg(\exp\left(\frac{4N\al2^{\frac{1}{N-1}}\|\nabla w\|_{L^N(\mb R^N)}^{\frac{N}{N-1}}}{{2N-2\beta-\mu}}\left(\frac{|w|}{\|\nabla w\|_{L^N(\mb R^N)}}\right)^{\frac{N}{N-1}}\right)\notag\\&\qquad\qquad\qquad\qquad\qquad\qquad-S_{N-2}\left(\frac{4N\al 2^{\frac{1}{N-1}}\|\nabla w\|_{L^N(\mb R^N)}^{\frac{N}{N-1}}}{{2N-2\beta-\mu}},\frac{|w|}{\|\nabla w\|_{L^N(\mb R^N)}}\right)\Bigg)dx\Bigg\}^{\frac{{2N-2\beta-\mu}}{2N}} \Bigg].
			\end{align}
			\noi	Now we choose $w$ with $\|w\|$ sufficiently small such  that $$\displaystyle\frac{4N\al2^{\frac{1}{N-1}}\|\nabla w\|_{L^N(\mb R^N)}^{\frac{N}{N-1}}}{{2N-2\beta-\mu}} <\al_N.$$ Then employing Theorem \ref{TM-ineq}  in \eqref{kc-MP1}, we obtain
			\begin{align}\label{1}
				\int_{{\mathbb R^N}}\left(\int_{{\mathbb R^N}} \frac{F(y,h(w))}{|y|^\beta|x-y|^{\mu}}dy\right)\frac{F(x,h(w))}{|x|^\beta}~dx 
				\leq C_2(N,\mu,\beta,\e) \left( \|w\|^{2N}  +  \|w\|^{2r} \right).
			\end{align}
		Using \eqref{energy}, \eqref{1}, $(V_1)$ and Lemma \ref{clm-mp}, we have
		\begin{align*}
			J(w) &= \frac1N  \int_{\mb R^N} |\nabla w|^N dx +\frac 1N \int_{\mb R^N} V(x) |h(w)|^N dx  -  \frac 12\int_{{\mathbb R^N}}\left(\int_{{\mathbb R^N}} \frac{F(y,h(w))}{|y|^\beta|x-y|^{\mu}}dy\right)\frac{F(x,h(w))}{|x|^\beta}~dx \\
			&\geq \frac1N C\| w\|^{N} -  \frac 12 C_2(N,\mu,\beta,\e) \left( \|w\|^{2 N}  +  \|w\|^{2r}
			\right).
		\end{align*}
		Now by taking  $r>0$ such that  $r>N$, we can choose {$0<\la_\beta<\min\{1,\la_0\}$(where $\la_0$ is defined in Lemma  \ref{clm-mp}  )}  sufficiently small so that, we finally obtain $J(w) \geq \tau_\beta >0$ for all $w\in W^{1,N}({\mathbb R^N})$ with $\|w\|=\la_\beta$ and for some $\tau_\beta>0$ depending on $\la_\beta$.\end{proof}
		\begin{lemma}\label{lem7.1.2}
		Let \eqref{be} hold and  the function $h$ be   as in \eqref{g}.	Assume that the conditions $(V_1)$, $(V_2^\prime)$ and $(f_1)$-$(f_5)$ hold. Then  there exists  a $w_\beta \in W^{1,N}({\mathbb R^N}) $ with $\|w_\beta\|>\la_\beta$ such that $ J(w_\beta)<0$ , where $\la_\beta$ is   given as in Lemma \ref{lem7.1}.
	\end{lemma}
	\begin{proof}
		The condition $(f_5)$ implies that there exist some positive constant $C_1, C_2>0$ such that
		\begin{equation}\label{new2}F(x,s) \geq C_1s^{\ell}-C_2\;\text{ for all}\; (x,s) \in{\mathbb R^N} \times [0,\infty).
		\end{equation}
		Let $\phi (\geq0) \in W^{1,N}({\mathbb R^N})$ such that $\|\phi\|=1$ and $supp(\phi)=\Om$. 
		Now by Lemma \ref{L1}-$(h_6)$,$(h_8)$ and \eqref{new2}, for large $t>1$, we obtain
		\begin{align*}
			&\int_{{\mathbb R^N}} \int_{{\mathbb R^N}} \frac{F(x, h(t \phi))F(y, h(t\phi
				))}{|x|^\beta|y|^\beta|x-y|^{\mu}}dxdy\\ &\geq \int_{{\Om}} \int_{{\Om}} \frac{(C_1 (h(t \phi(y)))^{\ell}-C_2)(C_1 (h(t \phi(x)))^{\ell}-C_2)}{|x-y|^{\mu}}~dxdy\\
			& = C_1^2  \int_{{\Om}} \int_{{\Om}} \frac{(h(t\phi(y)))^{\ell} (h(t\phi(x)))^{\ell}}{|x|^\beta|y|^\beta|x-y|^\mu}~dxdy \\
			& \quad -2C_1C_2\int_{{\Om}} \int_{{\Om}}\frac{(h(t\phi(x)))^{\ell}}{|x|^\beta|y|^\beta|x-y|^\mu}~dxdy + C_2^2 \int_{{\Om}} \int_{{\Om}} \frac{1}{|x|^\beta|y|^\beta|x-y|^{\mu}}~dxdy\\
			&\geq C_1^2 (h(1))^{2\ell} t^\ell \int_{{\Om}} \int_{{\Om}} \frac{(\phi(y))^{\frac \ell 2}(\phi(x))^{\frac \ell 2}}{|x|^\beta|y|^\beta|x-y|^{\mu}}~dxdy\\
			& \quad -2C_1C_2 2^{\frac{\ell}{2N}} t^{\frac \ell 2}\int_{{\Om}}  \int_{{\Om}}\frac{(\phi(x))^{\frac \ell 2}}{|x|^\beta|y|^\beta|x-y|^\mu}~dxdy + C_2^2 \int_{{\Om}} \int_{{\Om}} \frac{1}{|x|^\beta|y|^\beta|x-y|^{\mu}}~dxdy.
		\end{align*}
		Using the last relation together with \eqref{energy}  and applying Lemma \ref{L1}-$(h_5)$ we obtain
		\begin{align}\label{j}
			J(t \phi) & \leq \frac 1N \int_{\mb R^N} |\nabla t\phi|^N dx + \frac 1N \|V\|_{L^\infty(\Om)} \int_{\mb R^N} |h(t\phi)|^N dx - \frac{1}{2}\int_{{\mathbb R^N}} \int_{{\mathbb R^N}}\frac{F(x, h( t \phi))F(y, h(t \phi))}{|x|^\beta|y|^\beta|x-y|^{\mu}} dxdy\notag\\
			& \leq  \frac 1N\max\{1, \|V\|_{L^\infty(\Om)}\} \|t \phi\|^{N}- \frac{1}{2}\int_{{\mathbb R^N}} \int_{{\mathbb R^N}}\frac{F(x, h( t \phi))F(y, h(t \phi))}{|x|^\beta|y|^\beta|x-y|^{\mu}} dxdy\notag\\
			&\leq C_3 {t^N} -C_4 {t^\ell}
			+C_5 {t^{\frac \ell 2}} -{C_6},
		\end{align}
		where $C_i's$ are positive constants for $i=3,4,5,6$.
		From \eqref{j}, we infer that   $J(t\phi) \to -\infty$ as $t \to+\infty$, since $\ell>N$. Thus, there exists  $t_\beta(>0)\in\mathbb R$ so that $w_\beta(:=t_\beta \phi)\in W^{1,N}({\mathbb R^N})$ with $\|w_\beta\|> \la_\beta$ and $J(w_\beta)<0$.
	\end{proof}
	\section{Existence of   positive weak solutions}
	In this section, we give the detailed proofs of   Theorem \ref{T2} and Theorem \ref{T3}.
	\subsection{Compact-Coercive case}
	\noi Here we concentrate upon proving Theorem \ref{T2} with the potential function $V$ following  the assumptions $(V_1)$ and $(V_2)$. For that we define the energy functional $\mc E:=J|_E: E\to \mb R$ associated to \eqref{pp} as 
	\begin{align}\label{jin1}\mc E(w)=\frac 1N\int_{\mb R^N} |\nabla w|^N dx+\frac 1N \int_{\mb R^N} V(x)| h(w)|^N dx -\frac 12 \int_{{\mathbb R^N}}\left( \int_{{\mathbb R^N}}\frac{F(y,h(w))}{|y|^\beta|x-y|^\mu} dy\right) \frac{F(x,h(w))}{|x|^\beta} dx
	\end{align} for any $w\in  E.$
	We know that $\mc E\in C^1(E, \mb R)$. The following lemma shows that   $\mc E$ has the mountain pass geometry.
	\begin{lemma}\label{mpE}
		Let \eqref{be} hold and the function $h$ be defined as in \eqref{h-growth}.	Suppose that the conditions $(V_1)$, $(V_2)$ and $(f_1)$-$(f_5)$ hold.  Then we have the following assertions: 
		\begin{itemize}
			\item[$(i)$]  there exist $\la_\beta^*>0$ and $\tau_\beta^*>0$ such that	$\mc E(w)\geq \tau_\beta^*>0$ for any $ w\in E$  with  $\|w\|_E=\la_\beta^*.$
			\item[$(ii)$] there exists  a $w_\beta^* \in E $ with $\|w_\beta^*\|_E>\la_\beta^*$ such that $ \mc E(w_\beta^*)<0.$ 
		\end{itemize}
	\end{lemma}
	\begin{proof}
		With some minute modifications in the proof of lemmas \ref{mp1.0} and \ref{mp1.1} by replacing $\|\cdot\|$ by $\|\cdot\|_E$ the proofs of $(i)$ and $(ii)$ follows respectively.
	\end{proof}
	Next, we define the mountain pass  level
	\begin{align}\label{level11}\ds \theta_*:=\inf_{\gamma\in \uhat\Gamma}\max_{t\in[0,1]}\mc E(\gamma(t)),
	\end{align} where  $\ds \uhat\Gamma=\{\gamma\in C([0,1], E):\gamma(0)=0,\,\mc E(\gamma(1))<0\}$. By the previous lemma $\theta_*>0.$
	Now for some fixed $\delta>0$ we define the sequence of Moser functions $\{ \mc{\tilde {M}}_{k}\}$  as
	\begin{equation}\label{mos}
		\mc{\tilde {M}}_{k}(x,\delta)=\frac{1}{\omega_{N-1}^{\frac{1}{N}}}\left\{
		\begin{split}
			& (\log k)^{\frac{N-1}{N}}, \; \text{\;\; if \;} 0\leq |x|\leq \frac{\delta}{k},\\
			& \frac{\log \left(\frac{\delta}{|x|}\right)}{(\log k)^{\frac{1}{N}}}, \quad\text{\;\; if \;} \; \frac{\delta}{k}\leq |x|\leq \delta,\\
			& 0, \;\qquad\qquad\;\; \text{\;\; if \;} |x|\geq \delta.
		\end{split}
		\right.
	\end{equation}
	Then, supp$(	\mc{\tilde {M}}_{k}) \subset B_\delta(0)$ with $	\mc{\tilde {M}}_{k}(\cdot,\delta)\in W^{1,N}(\mb R^N)$ and . Moreover, $\ds\int_{\mb R^N} |\nabla 	\mc{\tilde {M}}_{k}(x,\delta)|^Ndx=1$ and $\ds\int_{\mb R^N} |	\mc{\tilde {M}}_{k}(x,\delta)|^Ndx$ $=O(\frac{1}{\log k})$ as $k\to+\infty.$ Also, $\ds\int_{\mb R^N}  V(x)	\left|h\left(\mc{\tilde {M}}_{k}(x,\delta)\right)\right|^Ndx$ $=O(\frac{1}{\log k})$ as $k\to+\infty.$ Set $$\mc{ {M}}_{k}(x,\delta):=\frac{\mc{\tilde {M}}_{k}(x,\delta)}{\|\mc{\tilde {M}}_{k}\|_E}.$$ By a simple computation, we can derive that 
	\begin{align*}
		\mc{ {M}}_{k}(x,\delta)^{\frac{N}{N-1}}=\frac{1}{\omega_{N-1}^{\frac{1}{N-1}}}\log k+a_k, \text{\;\; for all} \; |x|\leq \frac{\delta}{k},
	\end{align*} where $\{a_k\}$ is a bounded sequence of non-negative real numbers.
	\begin{lemma}\label{PS-level}
		Let \eqref{be} hold and the function $h$ be defined in \eqref{g}. If $V$ satisfies $(V_1)$, $(V_2)$ and $f$ satisfies the assumptions $(f_1)$-$(f_6)$,  then
		there exists some $k \in \mb N$ such that
		\begin{align}\label{key1}0<\theta_*<\frac{1}{2N }\left( \frac{{2N-2\beta-\mu}}{2N}\cdot\frac{\alpha_N}{\al_0}\right)^{N-1},\end{align} 
		where $\theta_*$ is defined in \eqref{level11}.
	\end{lemma}
	\begin{proof}
		We argue this proof by contradiction. Suppose \eqref{key1} doesn't hold.  Then  for given any $k \in \mb N,$ there exists  $t_k>0$ such that
		\begin{equation}\label{kc-PScond0}
			\begin{split}
				&\max_{t\in[0,\infty)} \mc E(t\mc M_k) = \mc E(t_k\mc M_k) \geq \frac{1}{2N } \left( \frac{{2N-2\beta-\mu}}{2N}\cdot\frac{\alpha_N}{\al_0}\right)^{N-1}.
			\end{split}
		\end{equation}
		Since  by  $(f_1)$, $F(x,h(t_k\mc M_k))\geq 0$ for all $k\in\mathbb N$, using  \eqref{jin1},  \eqref{kc-PScond0}, Lemma \ref{L1}-$(h_5)$ and the fact that $\|\mc M_k\|_E=1$, we get
		\begin{equation}\label{kc-PScond2}
			t_k^N \geq \frac{1}{2} \left( \frac{{2N-2\beta-\mu}}{2N}\cdot\frac{\alpha_N}{\al_0}\right)^{N-1}>0.
		\end{equation}
		Also, from \eqref{kc-PScond0}, it follows that
		$\frac{d}{dt}(J(t\mc M_k))|_{t=t_k}=0.$
		Combining this with Lemma \ref{L1}-$(h_4)$,  we obtain
		\begin{align}\label{kc-PS-cond3}
			t_k^N &= \int_{\mathbb R^N} \left(\int_{\mathbb R^N} \frac{F(y,h(t_k\mc M_k))}{|y|^\beta|x-y|^{\mu}}dy\right)\frac{f(x,h(t_k\mc M_k))}{|x|^\beta}h'(t_k\mc M_k)t_k\mc M_k ~dx\notag\\
			&\geq\frac 12 \int_{B_{\delta/k}(0)}\frac{f(x,h(t_k\mc M_k))}{|x|^\beta}h(t_k\mc M_k)\left( \int_{B_{\delta/k}(0)}\frac{F(y,h(t_k\mc M_k))}{|y|^\beta|x-y|^\mu}~dy\right)~dx.
		\end{align}
		By \eqref{h-growth},                                                                                                                                                                                                                                                                                                                                                                                                                                                                                                                                                                                                                                                                                                                                                                                                                                                                                                                                                                                                                                                                                                                                                                                                                                                                                                                                                                                                                                                                                                                                                                                                                                                                                                                                                                                                                                                                                                                                                                                                                                                                                                                                                                                                                                                                                                                                                                                                                                                                                                                                                                                                                                                                                                                                                                                                                                                                                                                                                                                                                                                                                                                                                                                                                                                                                                                                                                                                                                                                                                                                                                                                                                                                                                                                                                                                                                                                                                                                                                                                                                                                                                                                                                                                                                                                                                                                                                                                                                                                                                                                                                                                                                                                                                                                                                                                                                                                                                                                                                                                                                                                                                                                                                                                                                                                                                                                                                                                                                                                                                                                                                                                                                                                                                                                                                                                                                                                                                                                                                                                                                                                                                                                                                                                                                                                                                                                                                                                                                                                                                                                                                                                                                                                                                                                                                                                                                                                                                                                                                                                                                                                                                                                                                                                                                                                                                                                                                                                                                                                                                                                                                                                                                                                                                                                                                                                                                                                                                                                                                                                                                                                                                                                                                                                                                                                                                                                                                                                                                                                                                                                                                                                                                                                                                              for each $p>0$ there exists a constant $R_p$ such that
		\[sf(x,s)F(x,s) \geq p \exp\left( 2\al_0|s|^{\frac{2N}{N-1}}\right),\; \text{whenever}\; s \geq R_p.\]
		The relation \eqref{kc-PScond2} gives that
		${t_k}\mc M_k \to+\infty \;\text{as}\; k \to+\infty$ in $B_{\de/k}(0).$ Then  by Lemma \ref{L1}-$(h_7)$, for any given $\e>0$, there exists $k_0\in\mb N$ such that for all $k\geq k_0,$ $$|h({t_k}\mc M_k)|^{\frac{2N}{N-1}}\geq ({t_k}\mc M_k)^{\frac{N}{N-1}} (2^{\frac{1}{N-1}}-\e).$$ Moreover,   Lemma \ref{L1}-$(h_8)$ yields that  $h({t_k}\mc M_k) \to+\infty \;\text{as}\; k \to+\infty$, uniformly   in $B_{\de/k}(0).$   Hence, we can choose $s_p\in \mb N$ such that  in $B_{\de/k}(0)$,  it holds that
		\[h(t_k\mc M_k) \geq R_p,\; \text{for all}\; k\geq s_p.\]
		{On the other hand, using the same idea as  in \cite{yang-JDE} (see  equation $(2.11)$), we can deduce
			\[\int_{B_{\delta/k}(0)}\int_{B_{\delta/k}(0)} \frac{~dxdy}{|x|^\beta|y|^\beta|x-y|^\mu} \geq C_{\mu,\beta, N} \left(\frac{\delta}{k}\right)^{{2N-2\beta-\mu}},\]
		Using all these above estimates, from \eqref{kc-PS-cond3}, for sufficiently large $p>0$ and sufficiently large $k\in \mb N,$
		we get 
		\begin{align}\label{bd}
			t_k^N&\geq \frac p2 \int_{B_{\delta/k}(0)}\int_{B_{\delta/k}(0)}\exp \left( {2 \al_0(h(t_k\mc M_k))^{\frac{2N}{N-1}}}\right) \frac{~dxdy}{|x|^\beta|y|^\beta|x-y|^\mu}\notag\\
			&\geq \frac p2 \int_{B_{\delta/k}(0)}\int_{B_{\delta/k}(0)}\exp \left( {2\al_0 ({t_k}\mc M_k)^{\frac{N}{N-1}} (2^{\frac{1}{N-1}}-\e)}\right) \frac{~dxdy}{|x|^\beta|y|^\beta|x-y|^\mu}\notag\\
			&=\frac p2 \exp \left( \log k\left(\frac{2\al_0 (2^{\frac{1}{N-1}}-\e)  (t_k)^{\frac{N}{N-1}}}{{\omega}_{N-1}^{\frac{1}{N-1}}}\right)+ 2\al_0 (2^{\frac{1}{N-1}}-\e) t_k^{\frac{N}{N-1}}a_k\right)\int_{B_{\delta/k}(0)}\int_{B_{\delta/k}(0)}\frac{~dxdy}{|x|^\beta|y|^\beta|x-y|^\mu}\notag\\
			&\geq\frac p2 \exp \left( \log k\left(\frac{2\al_0 (2^{\frac{1}{N-1}}-\e) t_k^{\frac{N}{N-1}}}{{\omega}_{N-1}^{\frac{1}{N-1}}}\right)+ 2\al_0 (2^{\frac{1}{N-1}}-\e) t_k^{\frac{N}{N-1}}a_k\right)C_{\mu,\beta,N}\left(\frac{\delta}{k}\right)^{{2N-2\beta-\mu}}\notag\\
			&=\frac p2 \exp \left( \log k\left[\left(\frac{2\al_0 (2^{\frac{1}{N-1}}-\e) t_k^{\frac{N}{N-1}}}{{\omega}_{N-1}^{\frac{1}{N-1}}}\right)-({2N-2\beta-\mu})\right]+ 2\al_0 (2^{\frac{1}{N-1}}-\e) t_K^{\frac{N}{N-1}}a_k\right)C_{\mu,\beta,N}{\delta}^{{2N-2\beta-\mu}}.
		\end{align} Therefore, we have 
		\begin{align*}
			1&\leq \frac p2 {\delta}^{{2N-2\beta-\mu}}C_{\mu,\beta,N}\\&\qquad\qquad\exp \left( \log k\left[\left(\frac{2\al_0 (2^{\frac{1}{N-1}}-\e) t_k^{\frac{N}{N-1}}}{{\omega}_{N-1}^{\frac{1}{N-1}}}\right)-({2N-2\beta-\mu})\right]+ 2\al_0 (2^{\frac{1}{N-1}}-\e) t_K^{\frac{N}{N-1}}a_k-N\log t_k\right), 
		\end{align*} 
		which implies that the sequence $\{t_k\}$ must be bounded. If not, then $t_k\to+\infty$, which implies that the right hand side of the last inequality goes to $\infty,$ as $k\to+\infty$, which is absurd.  So, up to a subsequence, still denoted by $\{t_k\}$, $t_k\to \tilde t $ as $k\to+\infty$, for some $\tilde t\in \mb R.$ Now we claim that as $k\to+\infty,$
		\begin{equation}\label{cnvt}
			t_k^{\frac{N}{N-1}} \to \frac{1}{2^{\frac{1}{N-1}}} \left( \frac{{2N-2\beta-\mu}}{2N}\cdot\frac{\alpha_N}{\al_0}\right).
		\end{equation} Indeed, if not, then there exists some $\eta>0$ such that for sufficiently large $k\in\mb N,$ we have 
		$$t_k^{\frac{N}{N-1} }\geq \eta+ \frac{1}{2^{\frac{1}{N-1}}} \left( \frac{{2N-2\beta-\mu}}{2N}\cdot\frac{\alpha_N}{\al_0}\right).$$ By plugging  the last relation in \eqref{bd}, and using the fact that $\{a_k\}$ is a bounded sequence of non-negative real numbers, we obtain  
		\begin{align*}
				t_k^N&\geq	\frac p2 C_{\mu,\beta,N}{\delta}^{{2N-2\beta-\mu}}\notag\\&\qquad\exp \left( \log k\left[\left(\frac{2\al_0 (2^{\frac{1}{N-1}}-\e) \left(\eta+ \frac{1}{2^{\frac{1}{N-1}}} \left( \frac{{2N-2\beta-\mu}}{2N}\frac{\alpha_N}{\al_0}\right)\right)}{{\omega}_{N-1}^{\frac{1}{N-1}}}\right)-({2N-2\beta-\mu})\right]\right)\end{align*}\begin{align*}
				&\geq \frac p2 C_{\mu,\beta,N}{\delta}^{{2N-2\beta-\mu}}\notag\\&\qquad\exp \left( \log k\left[\left(\frac{2\al_0 (2^{\frac{1}{N-1}}-\e)\eta}{{({2N-2\beta-\mu}){\omega}_{N-1}^{\frac{1}{N-1}}}} +\frac { \left( {2}^{\frac{1}{N-1}}-\e\right)}{{2}^{\frac{1}{N-1}}}\right)-1\right]({2N-2\beta-\mu})\right).
		\end{align*}}Now by choosing $\e>0$ sufficiently small, using the relation $\al_N=N\omega_{N-1}^{\frac{1}{N-1}}$, by a simple computation, we can ensure that $$\left(\frac{2\al_0 (2^{\frac{1}{N-1}}-\e)\eta}{{({2N-2\beta-\mu}){\omega}_{N-1}^{\frac{1}{N-1}}}} +\frac { \left( {2}^{\frac{1}{N-1}}-\e\right)}{{2}^{\frac{1}{N-1}}}\right)-1>0$$ and thus, we achieve
		\begin{align*}
			\tilde	t^N&\geq \frac p2 C_{\mu,\beta,N}{\delta}^{{2N-2\beta-\mu}}\exp \left( \log k\left[\left(\frac{2\al_0 (2^{\frac{1}{N-1}}-\e)\eta}{{({2N-2\beta-\mu}){\omega}_{N-1}^{\frac{1}{N-1}}}} +\frac { \left( {2}^{\frac{1}{N-1}}-\e\right)}{{2}^{\frac{1}{N-1}}}\right)-1\right]({2N-2\beta-\mu})\right)\\
			& \to+\infty
		\end{align*} as $k\to+\infty,$ which is a contradiction. Therefore, \eqref{cnvt} holds.
		Now letting limit  $\e\to 0^+$ in the relation \eqref{bd}, we get
		\begin{align*}
			t_k^N
			&\geq\frac p2 \exp \left( \log k\left[\left(\frac{2\al_0 2^{\frac{1}{N-1}} t_k^{\frac{N}{N-1}}}{{\omega}_{N-1}^{\frac{1}{N-1}}}\right)-({2N-2\beta-\mu})\right]+ 2\al_0 (2^{\frac{1}{N-1}}-\e) t_k^{\frac{N}{N-1}}a_k\right)C_{\mu,\beta,N}{\delta}^{{2N-2\beta-\mu}},
		\end{align*}
		which together with \eqref{kc-PScond2}, and the fact that $\{a_k\}$ is a non-negative bounded sequence yields that
		$$t_k^N\geq \frac p2 C_{\mu,\beta,N}{\delta}^{{2N-2\beta-\mu}}. $$ Taking $k\to+\infty $ in the last relation, and using \eqref{cnvt}, we deduce
		$$\frac{1}{2^{\frac{1}{N-1}}} \left( \frac{{2N-2\beta-\mu}}{2N}\cdot\frac{\alpha_N}{\al_0}\right)\geq \frac p2 C_{\mu,\beta,N}{\delta}^{{2N-2\beta-\mu}}, $$ which is a contradiction, since $p>0$ can be chosen arbitrarily. This completes the proof of the lemma.
	\end{proof}
	\begin{proof}[\bf Proof of Theorem \ref{T2}] Since $\mc E$ satisfies Lemma \ref{mpE},  by mountain pass theorem,   there exists a Cerami sequence $ \{\uhat w_k\}\text{ in } E (\subset W^{1,N}(\mb R^N))$ for $\mc E$ at level $\theta_*$, that is, 
		\begin{align}\label{cs0} \mc E(\uhat w_k) \to \theta_*; \; \text{and}\; (1+\|\uhat w_k\|_E)\mc E^\prime(\uhat w_k) \to 0 \text{\;\; in \;} E^* \text{\;\; as $k\to+\infty$}.
		\end{align} 
		Also, Lemma \ref{mpE}-$(ii)$  yields that $\theta_*>0.$ Now by Lemma \ref{lem712}, we have that $\{\uhat w_k\}$ is bounded in $W^{1,N}(\mb R^N)$. Hence, up to a subsequence, still denoted by $\{\uhat w_k\}$,   there exists some $\uhat w\in W^{1,N}(\mb R^N)$ such that $\uhat w_k\rightharpoonup \uhat w$ weakly in $W^{1,N}(\mb R^N)$  as $k\to+\infty$. { Now using  Lemma  \ref{wk-sol}$-$\ref{kc-ws},
we deduce that $\uhat w$ satisfies \eqref{weak} for all $v\in W^{1,N}(\mb R^N)$ since $C_c^\infty(\mb R^N)$ is dense in $W^{1,N}(\mb R^N)$.} So, $\uhat w$ is a weak solution to \eqref{pp}.\\
		Next, we  show that  $\uhat w$ is nontrivial.	Suppose on the contrary, the weak solution we obtain $\uhat w\in W^{1,N}(\mb R^N)$ is trivial, that is $\uhat w\equiv 0$.  Then, from Lemma  \ref{PS-ws}, we have
		\begin{align}\label{F}
			\int_{\mathbb R^N} \left(\int_{\mathbb R^N} \frac{ F(y,h(\uhat w_k))}{|y|^\beta|x-y|^{\mu}}dy\right)\frac{F(x,h(\uhat w_k))} {|x|^\beta}~dx  \to 0\; \text{as}\; k \to+\infty,
		\end{align} where we used the fact that $h(\uhat w_k)\to 0$ strongly in $L^{\frac{2N^2}{2N-2\beta-\mu}}(\mb R^N)$, as $\frac{2N^2}{2N-2\beta-\mu}>N.$
		Since $E$ is compactly embedded into $L^N(\mb R^N)$, we have $\ds\int_{\mb R^N} V(x)|\uhat w_k|^N dx\to 0$ as $k\to+\infty$. This together  with
		\eqref{F} imply  that \begin{align*} \theta_*= \displaystyle\lim_{k \to+\infty}\mc E (\uhat w_k)= \frac{1}{N}\displaystyle\lim_{k\to+\infty}  \int_{\mb R^N}  |\nabla \uhat w_k|^N,
		\end{align*} that is,
		$$\lim_{k\to+\infty}\|\nabla \uhat w_k\|_{L^N(\mb R^N)}^N = \theta_*N.$$ Therefore, there exists a real number $l\in( 0,1)$, and corresponding to that $l$, there exists $k_0\in\mathbb N$ such that 
		\begin{align}\label{inq}
			\|\nabla \uhat w_k\|_{L^N(\mb R^N)}^{\frac{N}{N-1}} < \frac{{2N-2\beta-\mu}}{2N2^{\frac{1}{N-1}}}\cdot \frac{\al_N}{\al_0} (1-l), \text{\;\; for all}\;\; k\geq k_0. 
		\end{align} 
		Next, we show that
		\begin{align}\label{aaa}
			\int_{\mathbb R^N} \left(\int_{\mathbb R^N} \frac{ F(y,h(\uhat w_k))}{|y|^\beta|x-y|^{\mu}}dy\right)\frac{f(x,h(\uhat w_k))} {|x|^\beta} h'(\uhat w_k)w_k~dx  \to 0 \;\text{as}\; k \to+\infty.
		\end{align}  By the assumptions, $(f_2)$-$(f_3)$, for any $\e>0$, $r\geq N$,  there exist constants $ C>0$, $\al>\al_0>0$  such that
		\begin{align} 
			|f(x,s)s| \le \e |s|^N + C(r,\e) |s|^r \left[\exp\left(\al|s|^{\frac{2N}{N-1}}\right)-S_{N-2}(\al, s^2)\right],\;\; \text{for all}\; (x,s)\in {\mathbb R^N} \times \mb R.\label{f}
		\end{align}
		Now  using $(f_5)$, Proposition \ref{HLS} with $t=s, \beta=\vartheta$, \eqref{f}, Lemma \ref{techn}, Lemma \ref{mp-inq},   Lemma \ref{L1}-$(h_4)$ and H\"older's inequality, we deduce
		{\begin{align}\label{wk-soln}
				&\int_{{\mathbb R^N}}\left( \int_{\mathbb R^N}\frac{F(y,h(\uhat w_k))}{|y|^\beta|x-y|^{\mu}}dy\right)\frac{f(x,h(\uhat w_k))}{|x|^\beta}h'(\uhat w_k)\uhat w_k dx\notag\\
				&\leq\ell\int_{{\mathbb R^N}}\left( \int_{\mathbb R^N}\frac{f(y,h(\uhat w_k)) h(\uhat w_k)}{|y|^\beta|x-y|^{\mu}}dy\right)\frac{f(x,h(\uhat w_k))}{|x|^\beta}h(w_k) dx\notag \\
				&\leq \ell C(N,\mu,\beta) \left(\int_{{\mathbb R^N}}|f(x,h(\uhat w_k))h(\uhat w_k)|^{\frac{2N}{{2N-2\beta-\mu}}}~dx\right)^{\frac{{2N-2\beta-\mu}}{N}}\notag\\
				&  \leq  C(\ell,N,\mu,\beta,\e) \Bigg[ \|h(\uhat w_k)\|_{L^{\frac{2N^2}{{2N-2\beta-\mu}}}({\mathbb R^N})}^{2N}+ \|h(\uhat w_k)\|_{L^{\frac{2Nrp'}{{2N-2\beta-\mu}}}({\mathbb R^N})}^{2r} \Bigg(\int_{ {\mathbb R^N}}\Bigg(\exp\left( \frac{2N\al p}{{2N-2\beta-\mu}}|h(\uhat w_k)|^{\frac{2N}{N-1}}\right)\notag\\&\qquad\qquad\qquad\qquad\qquad\qquad\qquad\qquad\qquad-S_{N-2}\left(\frac{2N\al p}{{2N-2\beta-\mu}}, |h(\uhat w_k)|^2\right)\Bigg)dx \Bigg)^{\frac{{2N-2\beta-\mu}}{2Np}}\Bigg] \notag\\
				&  \leq C(\ell,N,\mu,\beta,\e,r,p)\Bigg[ \|h(\uhat w_k)\|_{L^{\frac{2N^2}{{2N-2\beta-\mu}}}({\mathbb R^N})}^{2N}\notag\\&\qquad\qquad\qquad+ \|h(\uhat w_k)\|_{L^{\frac{2Nrp'}{{2N-2\beta-\mu}}}({\mathbb R^N})}^{2r} \Bigg(\int_{ {\mathbb R^N}}\Bigg(\exp\left( \frac{2N\al p2^{\frac{1}{N-1}}\|\nabla\uhat  w_k\|_{L^N(\mathbb R^N)}^{\frac{N}{N-1}}}{{2N-2\beta-\mu}}\left(\frac{|\uhat w_k|}{\|\nabla\uhat w_k\|_{L^N(\mathbb R^N)}}\right)^{\frac{N}{N-1}}\right)\notag\\&\qquad\quad\qquad\qquad\qquad\qquad-S_{N-2}\left(\frac{2N\al p 2^{\frac{1}{N-1}}\|\nabla \uhat w_k\|_{L^N{(\mb R^N)}}^{\frac{N}{N-1}}}{{2N-2\beta-\mu}}, \left(\frac{|\uhat w_k|}{\|\nabla\uhat w_k\|_{L^N(\mathbb R^N)}}\right)\right)\Bigg)dx \Bigg)^{\frac{{2N-2\beta-\mu}}{2Np}}\Bigg].
		\end{align} }
		\noi Recalling \eqref{inq} and by choosing $p>1$ sufficiently close to $1$ and choosing $\al>\al_0$, very close to $\al_0$, we can have 
		$$\frac{2N\al p 2^{\frac{1}{N-1}}\|\nabla\uhat w_k\|_{L^N{(\mb R^N)}}^{\frac{N}{N-1}}}{{2N-2\beta-\mu}}<p\al \frac{\al_N}{\al_0}(1-l)<\al_N.$$
		Therefore,  in  view of  Theorem \ref{TM-ineq} and using \eqref{nv2}, for sufficiently large  $k\in\mathbb N$, from \eqref{wk-soln}, we get
		$$\int_{{\mathbb R^N}}\left( \int_{\mathbb R^N}\frac{F(y,h(\uhat w_k))}{|y|^\beta|x-y|^{\mu}}dy\right)\frac{f(x,h(\uhat w_k))}{|x|^\beta}h'(\uhat w_k)\uhat w_k dx<C,$$ where the constant $C$ is independent of $k$.
		Thus, by employing Vitali's convergence theorem, we obtain
		\[\int_{{\mathbb R^N}}\left( \int_{\mathbb R^N}\frac{F(y,h(\uhat w_k))}{|y|^\beta|x-y|^{\mu}}dy\right)\frac{f(x,h(\uhat w_k))}{|x|^\beta}h'(\uhat w_k)\uhat w_k~dx \to 0 \;\text{as}\; k \to+\infty.\]
		Since, $\{\uhat w_k\}$ is a Cerami sequence for $\mc E$,  we have $\displaystyle\lim_{k\to+\infty}\langle \mc E'(\uhat w_k), \uhat w_k \rangle_E=0$ which together with \eqref{aaa} implies that $$\int_{\mb R^N}|\nabla\uhat w_k|^N dx+ \int_{\mb R^N} V(x)|h(\uhat w_k)|^N dx \to 0 \text {\;\; as \;\;} {k\to+\infty}.$$ Invoking this and \eqref{F}  in \eqref{cs0}, we obtain $\theta_*=\displaystyle\lim_{k \to+\infty}\mc E(\uhat w_k)=0 $ which  contradicts  the fact that $\theta_*>0$. Thus, $\uhat w\not\equiv 0.$ 	Since $\uhat w$ is a weak solution to \eqref{pp}, it satisfies \begin{align*}
			\int_{\mathbb R^N} |\nabla \uhat w|^{N-2}\nabla \uhat w\nabla v~dx &+\int_{\mb R^N} V(x)|h(\uhat w)|^{N-2}h(\uhat w) h'(\uhat w) v dx\\& = \int_{\mathbb R^N} \left(\int_{\mathbb R^N} \frac{F(y,h(\uhat w))}{|y|^\beta|x-y|^{\mu}}dy\right)\frac {f(x,h(\uhat w))}{|x|^\beta}h'(\uhat w)v~dx,
		\end{align*} 
		for all $v \in W^{1,N}(\mb R^N)$. 
		Now choose $v = \uhat w^-:=\max\{-\uhat w,0\}\in W^{1,N}(\mb R^N)$ in the last equation. Then   $(f_1)$, $(V_1)$ and Lemma \ref{L1}-$(h_5)$ yield that  $\|\uhat w^-\|\leq0$ and hence, $\uhat w^-=0$ a.e. in ${\mathbb R^N}$. Therefore, $\uhat w\geq 0$ a.e. in ${\mathbb R^N}$.
		Now we can rearrange the  equation $\eqref{pp}$ as
		{\[-\Delta_N \uhat w+V(x)|\uhat w|^{N-2}\uhat w=V(x)\left(|\uhat w|^{N-2}w- |h(\uhat w)|^{N-2}h(\uhat w)h'(\uhat w)\right)+\left(\int_{\mathbb R^N} \frac{F(y,h(\uhat w))}{|y|^\beta|x-y|^{\mu}}dy\right)\frac{f(x,h(\uhat w))}{|x|^\beta}h'(\uhat w).\]}
		\noi By $(V_1)$, Lemma \ref{L1}-$(h_3),(h_5)$ and $(f_1)$ we infer  that the right hand side of this last equation is non negative. We claim that $\uhat w>0$. Suppose not, then there exists at least one point, say  $x_*\in \mb R^N$ such that $\uhat w(x_*)=0$.   Now Lemma \ref{reg} yields that $\ds \int_{\mb R^N} \frac{F(y,h(\uhat w))}{|y|^\beta|x-y|^{\mu}}dy\in L^\infty({\mathbb R^N})$. So, using the fact that $h'(\uhat w)\leq1$, and by recalling standard regularity results for the elliptic equations,  we infer that  $\tilde w \in L_{loc}^\infty({\mathbb R^N})\cap C_{loc}^{0,\xi}({{\mathbb R^N}})$ for some $\xi \in (0,1)$. Therefore, strong maximum principle implies that $\uhat w=0$ in $B_r(x_*)\subset \mb R^N$ for all $r>0$. So, $\uhat w\equiv0$ in $\mb R^N$ and this  contradicts  the fact that $\uhat w$ is non trivial.   Thus, $\uhat w>0$ and consequently, $h(\uhat w)>0$ which is a positive solution to \eqref{pq}. Hence, the proof of  Theorem \ref{T2} is complete.
	\end{proof}
	\subsection{Non-compact case} Here we aim to prove Theorem \ref{T3}. So, we consider the problem \eqref{pq} and the transformed equation \eqref{pp}, for $\beta=0$. The potential function $V$ satisfies $(V_1)$ and $(V_2^\prime)$. The condition $(V_2^\prime)$ is a weaker condition compared to $(V_2)$ and  induces lack of compactness in the embedding from $W^{1,N}(\mb R^N)$ into $L^N(\mb R^N)$. Therefore, we need to carry out  technically involved analysis to ensure the existence of nontrivial positive solution to \eqref{pp}.  For the brevity, we still denote the energy functional associated to \eqref{pp} with $\beta=0$ as $J:W^{1,N}(\mb R^N)\to\mb R$, which is defined as
	\begin{align}\label{modj}J(w)=\frac 1N\int_{\mb R^N} |\nabla w|^N dx+\frac 1N \int_{\mb R^N} V(x)| h(w)|^N dx -\frac 12 \int_{{\mathbb R^N}}\left( \int_{{\mathbb R^N}}\frac{F(y,h(w))}{|x-y|^\mu} dy\right) {F(x,h(w))} dx.
	\end{align} Clearly $J$ follows mountain pass geometry (that is, Lemma \ref{lem7.1} and Lemma \ref{lem7.1.2} with $\beta=0$) near  $0$.\\   Let $\ds \Gamma=\{\gamma\in C([0,1],W^{1,N}({\mathbb R^N})):\gamma(0)=0,J(\gamma(1))<0\}$. Define the mountain pass  level
	\begin{align}\label{level1}\ds \theta_c:=\inf_{\gamma\in \Gamma}\max_{t\in[0,1]}J(\gamma(t)).\end{align} 
	Then   by the mountain pass theorem, there exists a Cerami sequence $\{ w_k\}\subset W^{1,N}({\mathbb R^N})$ for $J$ at level $\theta_c$, that is, as $k\to+\infty$
	\begin{align}\label{cs111} J( w_k) \to \theta_c; \; \text{and}\; (1+\| w_k\|)J^\prime( w_k) \to 0 \text{\;\; in \;} \left(W^{1,N}({\mathbb R^N})\right)^*.
	\end{align} 
	Also by Lemma \ref{lem7.1}, we get $\theta_c>0.$ Moreover Lemma \ref{lem712} ensures that $\{ w_k\}$ is bounded in $W^{1,N}(\mb R^N)$ and hence, up to a subsequence, still denoted by $\{ w_k\}$,     $ w_k\rightharpoonup w$ weakly in $W^{1,N}(\mb R^N)$  as $k\to+\infty$ for some $ w\in W^{1,N}(\mb R^N)$. Now Lemma  \ref{wk-sol}$-$\ref{kc-ws} yield that $w$ satisfies\eqref{weak} for all $v\in W^{1,N}(\mb R^N)$. So, $ w$ is a weak solution to \eqref{pp}.   Our next goal is to show that  $w$ is non-trivial and positive. For that, we construct the arguments in the next lemmas, by assuming that $w\equiv0$ and finally, we arrive at some contradiction.\\ 
	We begin with analyzing the critical mountain pass level associated with the following newly defined functional
	\begin{align}\label{jin}J_\infty(w)=\frac 1N\int_{\mb R^N} |\nabla w|^N dx+\frac 1N \int_{\mb R^N} V_\infty| w|^N dx -\frac 12 \int_{{\mathbb R^N}}\left( \int_{{\mathbb R^N}}\frac{F(y,h(w))}{|x-y|^\mu} dy\right) {F(x,h(w))} dx\end{align} in the next lemma.  There we  consider the following equivalent norm on $W^{1,N}(\mb R^N)$, still denoted by $\|\cdot\|,$ defined as $$\|v\|=\left(\int_{{\mathbb R^N}}|\nabla v|^Ndx+ V_\infty\int_{{\mathbb R^N}} |v|^N dx\right)^{\frac 1p},\text{\;\;\; for } v\in W^{1,N}(\mb R^N). $$ One can easily see that $J_\infty\in C^1(W^{1,N}(\mb R^N), \mb R)$. Also, we can show that $J_\infty$ satisfies the mountain pass geometry as similar to  Lemma \ref{lem7.1} and Lemma \ref{lem7.1.2} with some  minute modifications in proofs of both the lemmas. 
	
\noi	For some fixed $\de>0$, let us set $$\mc{\widehat {M}}_{k}(x,\delta):=\frac{\mc{\tilde {M}}_{k}(x,\delta)}{\|\mc{\tilde {M}}_{k}\|},$$ where $\{ \mc{\tilde {M}}_{k}\}$ is the sequence of the Moser functions as defined in \eqref{mos}. 
	\begin{lemma}\label{PS-level1}
		Let  the conditions in Theorem \ref{T3} hold. Then,  then 
			there exists some $k \in \mb N$ such that
		\begin{align}\label{key}\max_{t\in[0,\infty)} J_\infty(t\mc{ {\widehat {M}}}_k) < d_c:=\frac{1}{2N }\left( \frac{{2N-\mu}}{2N}\cdot\frac{\alpha_N}{\al_0}\right)^{N-1}.\end{align} 
	\end{lemma}
	\begin{proof}
The proof of the lemma follows  in a similar manner as in the proof of Lemma \ref{PS-level} by taking $\beta=0$.
			\end{proof}
	\begin{remark}
		The above lemma gives that $\theta_c<d_c:=\frac{1}{2N }\left( \frac{{2N-\mu}}{2N}\cdot\frac{\alpha_N}{\al_0}\right)^{N-1}$, where $\theta_c$ is defined in \eqref{level1}. Indeed, using  $(V_2^\prime)$  we get that $J(w)\leq J_\infty (w)$ for all $w\in W^{1,N}(\mb R^N)$. Then by using Lemma \ref{PS-level} we get the desired  estimate.
	\end{remark}
		\noi In the next lemma, we show the non-vanishing behaviour of the Cerami sequence $\{w_k\}$, which is defined  in \eqref{cs111}.
	\begin{lemma}\label{non-van}
		Let  the conditions in Theorem \ref{T3} hold.	Then, there exist positive constants $b,R$ and a sequence $\{z_k\}\subset \mb R^N$ such that \begin{align}\label{nv}
			\lim_{k\to+\infty}\int_{B_R(z_k)} |h(w_k)|^N\geq b>0.
		\end{align} 
	\end{lemma}
	\begin{proof}
		Suppose on the contrary, \eqref{nv} does not hold. Then we have
		\begin{align*}
			\lim_{k\to+\infty}\sup_{z\in\mb R^N}\int_{B_R(z)} |h(w_k)|^N=0.
		\end{align*}
		This, together with Lions compactness lemma (see Lemma {I.1} in \cite{Lions}) implies that, as $k\to+\infty$
		\begin{align}\label{nv2}
			h(w_k)\to 0 \text{\;\; strongly in } L^q(\mb R^N), \text{\;\; for all \;\;} q\in (N,\infty).
		\end{align}
		\noi Therefore, $h(w_k(x))\to 0$ a.e. in $\mb R^N.$ Using this together with  $(f_2)$, for sufficiently large $k\in\mathbb N$, we have 
		$$ |F(x, h(w_k))|\leq C |h( w_k)|^N $$  for some positive constant $C>0$.  So, by \eqref{k2} with $\beta=0$, we get 
		\begin{align}\label{F1}
			\int_{\mathbb R^N} \left(\int_{\mathbb R^N} \frac{ F(y,h(w_k))}{|x-y|^{\mu}}dy\right){F(x,h(w_k))} ~dx \leq C(N,\mu,\beta)\|h(w_k)\|_{L^{\frac{2N^2}{2N-\mu}}(\mb R^N)}^{2N} \to 0\; \text{as}\; k \to+\infty,
		\end{align} in the last line we used the fact that $h(w_k)\to 0$ strongly in $L^{\frac{2N^2}{2N-\mu}}(\mb R^N)$, since $\frac{2N^2}{2N-\mu}>N.$
		This yields that \begin{align*} \theta_c= \displaystyle\lim_{k \to+\infty}J(w_k)= \frac{1}{N}\displaystyle\lim_{k\to+\infty} \left( \|\nabla w_k\|_{L^N(\mb R^N)}^N+\int_\Om V(x) |h(w_k)|^N dx\right)\geq \frac 1N \displaystyle\lim_{k\to+\infty}  \|\nabla w_k\|_{L^N(\mb R^N)}^N.
		\end{align*} That is,
		$$\lim_{k\to+\infty}\|\nabla w_k\|_{L^N(\mb R^N)}^N \leq \theta_cN.$$
		Now using the similar arguments, we get \eqref{aaa} for $\beta=0$. Using this  and the fact that $\ds\lim_{k\to+\infty}\langle J(w_k), w_k\rangle=0$, we obtain
		$$\int_{\mb R^N}|\nabla w_k|^N dx+ \int_{\mb R^N} V(x)|h( w_k)|^N dx \to 0 \text {\;\; as \;\;} {k\to+\infty}.$$ Plugging this and \eqref{F1}  in \eqref{cs111}, it follows that $\theta_c=\displaystyle\lim_{k \to+\infty}J(w_k)=0 $ which is a contradiction to the fact that $\theta_c>0$. The proof is completed.
	\end{proof}
	\noi Next, we define the functional $I_\infty:W^{1,N}(\mb R^N)\to \mb R$ as 
	\begin{align}\label{iinft}
		I_\infty(w)=\frac 1N\int_{\mb R^N} |\nabla w|^N dx-\frac 1N \int_{\mb R^N} V_\infty|h (w)|^N dx -\frac 12 \int_{{\mathbb R^N}}\left( \int_{{\mathbb R^N}}\frac{F(y,h(w))}{|x-y|^\mu} dy\right) {F(x,h(w))} dx.
	\end{align} Clearly $I_\infty\in C^1\left(W^{1,N}(\mb R^N),\mb R\right).$
	\begin{lemma} 
		Let  the conditions in Theorem \ref{T3} hold. Then,	the sequence $\{w_k\}$ is a Palais-Smale sequence for $I_\infty$, where $\{w_k\}$ is defined as in \eqref{cs111}.
	\end{lemma}
	\begin{proof}
		By $(V_2^\prime)$, for any given $\e>0$, there exists a constant $l>0$ such that for all $|x|\ge l$, we have
		$$|V(x)-V_\infty|<\e,$$ which implies 
		\begin{align*}
			|I_\infty(w_k)-J(w_k)|&=\frac 1N \int_{B_C(0)}|V(x)-V_\infty||h (w_k)|^N dx+\frac 1N \int_{\mb R^N\setminus B_r(0)}|V(x)-V_\infty||h (w_k)|^N dx\\
			&\leq \frac 1N \max_{x\in B_l(0)} |V(x)-V_\infty|\int_{B_l(0)}|h (w_k)|^N dx+\frac 1N \e \int_{\mb R^N\setminus B_l(0)}|h (w_k)|^N dx\\&\leq o_k(1),
		\end{align*}where in the last inequality we used the fact that  the embedding $W^{1,N}(B_l(0))\hookrightarrow\hookrightarrow L^N_{loc}(B_l(0))$ is compact. Thus $$I_\infty(w)\to \theta _c \text{\;\; as \;\;} k\to+\infty.$$ Arguing similarly, we get
		\begin{align*}\sup_{\zeta \in W^{1,N}(\mb R^N),\; \|\zeta\|\leq 1}|\langle I'_\infty(w_k)-J'(w_k), \zeta\rangle|&=\sup_{ \zeta \in W^{1,N}(\mb R^N),\; \|\zeta\|\leq 1}\left |\int_{\mb R^N} (V(x)-V_\infty)|h(w_k)|^{N-2}h(w_k)h'(w_k) \zeta dx \right|\\&=o_k(1). 
		\end{align*} Hence $I'_\infty\to 0$ in $\left( W^{1,N}(\mb R^N)\right)^*$ as $k\to+\infty$. This proves the lemma. 
	\end{proof}
	\noi Now we define ${\tilde w_k}(x)=w_k(x+z_k)$, where the sequences $\{z_k\}$ and $\{w_k\}$ are defined in Lemma \ref{non-van} and in \eqref{cs111}, respectively. Then  $\{{\tilde w_k}\}$ is bounded in $W^{1,N}(\mb R^N)$ and hence, there is some $\tilde w$ such that up to a subsequence,  ${ \tilde w_k}\rightharpoonup \tilde w$ as $k\to+\infty$. Also,  it can easily be computed that,  as $k\to+\infty$,
	$$I_\infty({\tilde w_k})=I_\infty({w_k})\to \theta_c\text{\;\;\; and \;\;} I'_\infty({\tilde w_k})\to 0 \text{\;\;\; in \;\;} \left(W^{1,N}(\mb R^N)\right)^*.$$ Therefore, following the previous arguments, we get that $\tilde w$ is a critical point of the functional $I_\infty$.
	Thus, using Lemma \ref{non-van} and Lemma \ref{L1}-$(h_5)$, we get
	{\[\int_{B_R(0)}|\tilde{w}|^N dx=\displaystyle \lim_{k\to+\infty}\int_{B_R(0)}|\tilde{w}|^N dx=\displaystyle \lim_{k\to+\infty}\int_{B_R(z_k)}|{w}|^N dx\ge\displaystyle \lim_{k\to+\infty}\int_{B_R(z_k)}|h({w})|^N dx\geq b>0.\]}
	Hence, $\tilde w \not \equiv 0$.\\
	{ \begin{lemma}\label{last}
			Suppose  the conditions in Theorem \ref{T3} hold. Let $\theta _\infty:=\displaystyle\inf_{\gamma\in \Gamma_\infty}\max_{t\in[0,1]} I_\infty(\gamma(t))$ and $\ds \Gamma_\infty:=\{\gamma\in C([0,1],W^{1,N}({\mathbb R^N})):\gamma(0)=0,\,I_\infty(\gamma(1))<0\}.$ Then, \begin{align}\label{comp}
				\theta_\infty\leq I_\infty(\tilde w)\leq \theta_c,\end{align} where $\theta _c$ is defined in \eqref{level1}.
	\end{lemma}}
	\begin{proof}
		First, we show that $I_\infty(\tilde w)\leq \theta_c$.  For that, using  Lemma \ref{L1}-$(h_4)$ and $(f_5)$, we obtain
		\begin{align*}\frac 1N {f(x,h(\tilde w_k))} h'(\tilde w_k) \tilde w_k-\frac 12{F(x,h(\tilde w_k))}&\geq \frac {1}{2N} {f(x,h(\tilde w_k))} h(\tilde w_k) -\frac 12{F(x,h(\tilde w_k))}\geq0.
		\end{align*} This together with Lemma \ref{L1}-$(h_8)$ and
		Fatou's lemma implies that
		\begin{align*}
			\theta_c&=\ds \lim \inf _{k\to+\infty}\left(I_\infty(\tilde w_k)-\frac{1}{N}\langle I_\infty'(\tilde w_k), \tilde w_k\rangle\right)\notag\\
			&=\ds \lim \inf _{k\to+\infty}\bigg(\int_{\mb R^N}\frac1N V_\infty \left(|h(\tilde w_k)|^N - |h(\tilde w_k)|^{N-2}h(\tilde w_k)h'(\tilde w_k) \tilde w_k\right) dx\notag\\&\;\;\;\;\;\;+\int_{\mathbb R^N} \left(\int_{\mathbb R^N} \frac{F(y,h(\tilde w_k))}{|x-y|^{\mu}}dy \right)\left[\frac 1N {f(x,h(\tilde w_k))} h'(\tilde w_k) \tilde w_k-\frac 12{F(x,h(\tilde w_k))}\right] dx\bigg)\notag\\
			&\geq \int_{\mb R^N}\frac1N V_\infty \left(|h(\tilde w)|^N - |h(\tilde w)|^{N-2}h(w)h'(\tilde w) \tilde w\right) dx\notag\\&\;\;\;\;\;\;+\int_{\mathbb R^N} \left(\int_{\mathbb R^N} \frac{F(y,h(\tilde w))}{|x-y|^{\mu}}dy \right)\left[\frac 1N {f(x,h(\tilde w))} h'(\tilde w) \tilde w-\frac 12 {F(x,h(\tilde w))}\right] dx\notag\\
			&=I_\infty(\tilde w)-\frac{1}{N}\langle I_\infty'(\tilde w), \tilde w\rangle\\&=I_\infty(\tilde w),
		\end{align*}
		since  $\tilde w$ is a critical point of the functional $I_\infty$. Hence, $I_\infty(\tilde w)\leq \theta_c$.\\
		Next,  we use the approach as in \cite{do} for showing $I_\infty(\tilde w)\geq \theta_\infty$. For that,  we construct the following arguments to achieve a suitable path $\gamma:[0,1]\to W^{1,N}(\mb R^N)$ such that  
		\begin{equation}\label{path}
			\left\{
			\begin{split}
				\gamma(0)&=0,\;\;\; I_\infty(\gamma(1))<0,\;\;\; \tilde w\in \gamma([0,1]),\\
				&\ds\max_{t\in[0,1]}I_\infty(\gamma(t))	=I_\infty(\tilde w).
			\end{split}
			\right.
		\end{equation} So, we define
		\begin{equation*}
			{\tilde {w_t}}(x):=\left\{
			\begin{split}
				& \tilde w(x/t),\text{\;\;\;\; if \;\;} t>0,\\
				& 0,\; \text{\;\;\;\;\qquad\;\; if \;\;} t=0.
			\end{split}
			\right.
		\end{equation*} Our next aim is to choose the real numbers $0<t_1<1<t_2<s_0$ such that the path $\gamma$ defined by three pieces in the below is turned out to be our desired path:
		\begin{equation*}
			\gamma(s)=	\left\{
			\begin{split}
				&s{\tilde w_{t_1}}(x),\;\;\; \text{\;\;\;\; if \;\;} s\in[0,t_1],\\
				&s\tilde w_{t}(x),\;\;\; \text{\;\;\;\; if \;\;} s\in[t_1,t_2],\\
				&s\tilde w_{t_2}(x),\;\;\; \text{\;\;\;\; if \;\;} s\in[t_2,s_0].
			\end{split}
			\right.
		\end{equation*}
		Let us define the function $g:\mb R^N\times\mb R\to \mb R$ as $$g(x,s):=h'(s)\left[-V_\infty|h(s)|^{N-2}h(s)+\left(\int_{\mb R^N}\frac{F(y,h(s))}{|x-y|^\mu}dy\right){f(x,h(s))}\right].$$ Since $\tilde w$ is a critical point of $I_\infty,$  $\tilde w$ is solution to the following problem 
		\begin{align}\label{ppp}
			-\Delta_N w=g(x,w) \text{\;\;\; in \;\;} \mb R^N
		\end{align} and $\tilde w\not\equiv0$. Thus, we have $$\int_{\mb R^N} g(x,\tilde w)\tilde w dx= \|\nabla \tilde w\|_{L^N(\mb R^N)}^N>0.$$ Thus, there exists some $s_0>1$ such that  $$\int_{\mb R^N} g(x,s\tilde w)\tilde w dx>0, \;\;\;\forall s\in [1,s_0].$$ Set $\Phi(x, s):=g(x,s)/s^{N-1}.$ Then, it is easy to see that $\ds\lim_{|s|\to 0} \Phi(x,s)=-c<0$ uniformly in $x\in\mb R^N$ for some $c>0$ and $\Phi \in C(\mb R^N\times\mb R,\mb R)$ and \[\int_{\mb R^N} \Phi(x,s\tilde w)|\tilde w|^N dx>0, \;\;\;\forall s\in [1,s_0].\]
		On the other hand, 
		\begin{align}\label{mon}
			\frac{d}{ds}I_\infty(s\tilde w_t)&=\langle I'(s\tilde w_t), \tilde w_t\rangle\notag\\
			&=s^{N-1}\left(\|\nabla \tilde w_t\|_{L^N(\mb R^N)}^N-\int_{\mb R^N}g(x, s\tilde w_t)\tilde w_t dx\right)\notag\\
			&=s^{N-1}\left(\|\nabla \tilde w\|_{L^N(\mb R^N)}^N-t^N\int_{\mb R^N}\Phi(x, s\tilde w)|\tilde w|^N dx\right).
		\end{align} Therefore, we can choose $t_1\in (0,1)$ sufficiently small and $t_2>1$ sufficiently large such that from the last relation, we obtain
		\begin{align*}
			\|\nabla \tilde w\|_{L^N(\mb R^N)}^N-t_1^N\int_{\mb R^N}\Phi(x, s\tilde w)|\tilde w|^N dx>0, \text{\;\;\; for all \;\;} s\in [1,s_0]
		\end{align*} and 
		\begin{align}\label{t2}
			\|\nabla \tilde w\|_{L^N(\mb R^N)}^N-t_2^N\int_{\mb R^N}\Phi(x, s\tilde w)|\tilde w|^N dx\leq -\frac{1}{s_0-1}\|\nabla\tilde w\|_{L^N(\mb R^N)}^N, \text{\;\;\; for all \;\;} s\in [1,s_0].
		\end{align}
		Thus from \eqref{mon}, it follows that $I_\infty(s\tilde w_{t_1})$ is increasing on $[0,t_1]$
		and one can check that $I_\infty(s\tilde w_{t_1})$ takes its maximum value at $s=1$. 
		In addition, {from Theorem \ref{TM-ineq}, we have $f(\cdot,h(w))h(w),\; F(x,h(w)) \in L_{loc}^q({\mathbb R^N})$, for $1\leq q<\infty$.
			{ Since by Lemma \ref{reg} for $\beta=0$,  $\ds\int_{\mb R^N} \frac{F(y,h(w))}{|x-y|^{\mu}}dy\in L^\infty({\mathbb R^N})$, using  $h'(w)\leq1$, we obtain $$\left(\int_{\mathbb R^N} \frac{F(y,h(w))}{|x-y|^{\mu}}~dy \right){f(x,h(w))} h'(w)\in L_{loc}^q({\mathbb R^N}),$$ for $1 \leq q <\infty$. Now by recalling regularity results for the elliptic equations, from \eqref{ppp}, we infer that  $\tilde w \in L_{loc}^\infty({\mathbb R^N})\cap C_{loc}^{1,\xi}({{\mathbb R^N}})$ for some $\xi \in (0,1)$.}}
		Therefore, using Proposition \ref{Pohzv} by taking $p=N$ and $\beta=0$, we have $\ds\int_{\mb R^N} G(x,\tilde w)dx=0,$ where $G(x,t)=\ds\int_0^t g(x, s) ds$ is the primitive of $g$. Therefore, we get
		\[I_\infty(\tilde w_t)=I_\infty(\tilde w)=\frac 1N\|\nabla\tilde w\|_{L^N(\mb R^N)}^N. \] This together with \eqref{t2} yields that
		\begin{align*}
			I_\infty(s_0\tilde w_{t_2})&=I_\infty(\tilde w_{t_2})+\int_1^{s_0} \frac{d}{ds}I_\infty(s\tilde w_{t_2})ds\\
			&=\frac 1N\|\nabla\tilde w\|_{L^N(\mb R^N)}^N-\int_1^{s_0}\frac{1}{s_0-1}\|\nabla\tilde w\|_{L^N(\mb R^N)}^N ds\\
			&<\left(\frac 1N -1\right)\|\nabla\tilde w\|_{L^N(\mb R^N)}^N<0. 
		\end{align*}
		Thus we have achieved our desired path, which together with the definition of $\theta_\infty$ gives that
		\[\theta_\infty\leq \max_{t\in[0,1]} I_\infty (\gamma(t))=I_\infty(\tilde w).\] Hence, the relation  \eqref{comp} hold. This completes the lemma.
	\end{proof}
	\noi {\bf Proof of Theorem \ref{T3}}:  Consider the path $\gamma$ defined in \eqref{path}. Since $\gamma\in \Gamma _\infty\subset\Gamma,$ $ \gamma(t)(x)>0$ and by $(V_2^\prime)$, we have $V(x)\leq V_\infty$ with $V\not=V_\infty,$ from \eqref{level1}, \eqref{comp}, and \eqref{path}, we obtain
	\begin{align*}
		\theta_c&\leq\ds \max_{t\in[0,1]}J(\gamma(t))=J(\gamma(t_{max}))\\&<I_\infty(\gamma(t_{max}))\leq \max_{t\in[0,1]}I_\infty(\gamma(t))\\&=I_\infty(\tilde w)\leq \theta _c.
	\end{align*}This gives a contradiction. Thus $w$ is non-trivial weak solution to \eqref{pp}. Next, we prove that $w>0$ in ${\mathbb R^N}$.
	Since $w$ is a weak solution to \eqref{pp}, it satisfies \begin{align*}
		\int_{\mathbb R^N} |\nabla w|^{N-2}\nabla w\nabla \varphi~dx &+\int_{\mb R^N} V(x)|h(w)|^{N-2}h(w) h'(w)\varphi dx\\& = \int_{\mathbb R^N} \left(\int_{\mathbb R^N} \frac{F(y,h(w))}{|x-y|^{\mu}}dy\right){f(x,h(w))}h'(w)\varphi~dx,
	\end{align*} 
	for all $\varphi \in W^{1,N}({\mathbb R^N})$.
	Now in particular, taking $\varphi = w^-:=\max\{-w,0\}$ in the last equation, and using $(f_1)$, $(V_1)$ and Lemma \ref{L1}-$(h_5)$, we obtain $\|w^-\|\leq0$ which implies that $w^-=0$ a.e. in ${\mathbb R^N}$. Therefore, $w\geq 0$ a.e. in ${\mathbb R^N}$.
	Now rearranging the equation $\eqref{pp}$, we get
	{\[-\Delta_N w+V(x)|w|^{N-2}w=V(x)\left(|w|^{N-2}w- |h(w)|^{N-2}h(w)h'(w)\right)+\left(\int_{\mathbb R^N} \frac{F(y,h(w))}{|x-y|^{\mu}}dy\right)f(x,h(w))h'(w)\]}
	\noi and right hand side this equation is non negative, thanks to $(V_1)$, Lemma \ref{L1}-$(h_3),(h_5)$ and $(f_1)$. { Recall that in the proof of Lemma \ref{last}, we obtain $w \in L_{loc}^\infty({\mathbb R^N})\cap C_{loc}^{1,\xi}({{\mathbb R^N}})$ for some $\xi \in (0,1)$.} Therefore, using the strong maximum principle we get  $w\equiv0$ in $\mb R^N$. Hence, $h(w)>0$ which serves as a positive  solution to \eqref{pq}. This completes the proof of Theorem \ref{T3}. \qed
	\section{Appendix}
	\noi Here, we prove the generalized Pohozaev identity for the following quasilinear equation:
	\begin{equation*}\label{qq}
		\tag{$Q_{*}$}\left\{
		\begin{array}{l}
			-\Delta_p w+|h(w)|^{p-2}h(w)h'(w) =
			\left(\displaystyle\int_{\Om}\frac{F(y,h(w))}{|y|^\beta|x-y|^{\mu}}~dy\right)\frac{f(x,h(w))}{|x|^\beta}h'(w)\; \text{\;\;\; in}\;
			\mathbb R^N,
		\end{array}
		\right. 
	\end{equation*} 
	where $1<p\leq N,$  $0<\mu<N, \beta\geq 0,$ and $2\beta+\mu\leq N.$
	The nonlinearity  $f:\mb R^N\times \mb R\to \mb R$ is a continuous function and $F(x,s)=\int_{0}^s f(x,t)dt$ is the primitive of  $f$.  Similar results for the semilinear Choquard equations  with Laplacian operator, readers are referred to \cite{gao-yang-du, moroz4}.
	\begin{proposition}\label{Pohzv} Let $w\in W^{1,p}(\mb R^N)\cap L_{loc}^\infty(\mb R^N) $ be a weak solution to problem \eqref{qq}. Then
		\begin{align*}
			\frac{(N-p)}{p}\int_{\mb R^N} |\nabla w|^p\;dx+\frac{N}{p} \int_{\mb R^N} |h(w)|^p dx+\frac{(2N-\mu-2\beta)}{2}\int_{{\mathbb R^N}}\left(\int_{{\mathbb R^N}} \frac{F(y,h(w))}{|y|^\beta|x-y|^\mu}dy\right)\frac{F(x,h(w))}{|x|^\beta}dx=0.
		\end{align*}
	\end{proposition}
	\begin{proof}
		{	Using the idea of \cite{gv}, let \(w_{\e}\) be the classical solution in \(C^{3}_{loc}(\mb R^N\setminus \{0\})\) of  }
		\begin{equation*}
			\left\{
			\begin{array}{l}
				-\text{div}(\e+|\nabla w_\e|^2)^{\frac{p-2}{2}} w_{\e} 
				+|h(w_\e)|^{p-2}h(w_\e)h'(w_\e)=
				\left(\displaystyle\int_{\Om}\frac{F(y, h(w_\e))}{|y|^\beta|x-y|^{\mu}}~dy\right)\frac{f(x, h(w_\e))}{|x|^\beta} h'(w_\e)\; \text{\;\;\; in}\;
				\mathbb R^N.
			\end{array}
			\right. 
		\end{equation*} 
		Then $w_{\e}$ is bounded in $C^{1,\theta}(\mb R^N\setminus \{0\})$ independently of $\e\in (0,1]$ and converges to $w$ in $C^{1,\tilde{\theta}}(\mb R^N\setminus \{0\})$ for any $\theta< \tilde\theta$ as $\e\ra 0$. 
		For 0<r<R, define $\phi_{r,R}\in C_{c}^{\infty}(\mathbb R^N)$ with $0\leq \phi_{r,R}\leq 1$,  $\phi_{r,R}(x)=0$ in $|x|<\frac{r}{2}$ and   $\phi_{r,R}(x)=1$ on $r<|x|< R$ with $|\na \phi_{r,R}(x)|<\frac{c}{R}$. By testing the equation against the function $\psi_{r, R}(x)= \phi_{r,R} (x \cdot \nabla w_\e(x))$ and integrate over $\mb R^N$, we have
		\begin{align*}\int_{\mb R^N} (\e+|\na w_{\e}|^2)^{\frac{p-2}{2}}\na w_{\e}\cdot \na \psi_{r,R} dx& + \int_{\mb R^N} |h(w_\e)|^{p-2}h(w_\e )h'(w_\e) \psi_{r,R} dx\\&=
			\int_{\mb R^N}\left(\int_{\mb R^N}\frac{F(y,h(w_\e))}{|y|^\beta|x-y|^{\mu}}dy\right) \frac{f(x,h(w_\e))}{|x|^\beta} h'(w_\e) \psi_{r,R}(x)  dx.\end{align*}
		Now we compute, for every $R>r>0$,
		\begin{align*}
			\int_{\mb R^N} |h(w_\e)|^{p-2} h(w_\e)h'(w_\e) \psi_{r,R}(x) dx & = \int_{\mb R^N} |h(w_\e)|^{p-2} h(w_\e) h'(w_{\e}) \phi_{r,R}(x) x. \na w_{\e}(x) dx \\
			&= \int_{\mb R^N} \phi_{r,R}(x) x. \na\left(\frac{|h(w_\e)|^p}{p}\right)(x) dx\\
			&= - \int_{\mathbb R^N} (N \phi_{r,R}(x)+ \la x\cdot \nabla \phi_{r,R}) \frac{|h(w_\e)|^p}{p} dx.
		\end{align*}
		Letting $\e\ra 0$, $r \ra 0$ and $R\ra \infty$, we obtain
		\[  \int_{\mb R^N} |h(w_\e)|^{p-2} h(w_{\e})h'(w_\e)\psi_{r,R}(x) dx \ra -\frac{N}{p} \int_{\mb R^N } |h(w)|^p dx,\]
		thanks to Lebesgue's dominated convergence Theorem.
		Next, we have 
		\begin{align*}
			&\int_{\mb R^N} |(\e+|\na w_{\e}(x)|^2)^{\frac{p-2}{2}} \na w_{\e}(x)\cdot \na \psi_{r,R}(x) dx\\ &= \int_{\mb R^N} (\e+|\na w_{\e}(x)|^2)^{\frac{p-2}{2}} \na w_{\e}(x) \na (\phi_{r,R}(x) x. \na w_{\e}(x)) dx \\
			&= \int_{\mb R^N} \phi_{r,R}(x) (\e+ |\na w_{\e}|^2)^{\frac{p-2}{2}} |\na w_{\e}(x)|^p - \frac{N}{p}\int_{\mb R^N} \phi_{r,R}(x) (\e+ |\na w_{\e}|^2)^{\frac{p}{2}} 
			- \int_{\mb R^N} x. \na \phi_{r,R}(x) (\e+ |\na w_{\e}|^2)^{\frac{p}{2}} dx.
		\end{align*}
		Similarly,	taking $\e\ra 0$ in the left hand side of the last relation,
		\begin{align*}
			\int_{\mb R^N} |(\e+|\na w_{\e}|^2)^{\frac{p-2}{2}} \na w_{\e}(x)\cdot \na \psi_{r,R}(x) dx\ra  -
			\int_{\mathbb R^N} ((N-p) \phi_{r,R}(x)+ x\cdot \nabla \phi_{r,R}( x)) \frac{|\na w(x)|^p}{p} dx
		\end{align*}
		and using this,
		\[\lim_{r \ra 0}\lim_{R\ra \infty}  \int_{\mb R^N}  |(\e+|\na w_{\e}(x)|^2)^{\frac{p-2}{2}} \na w_{\e}(x)\cdot \na \psi_{r,R}(x) dx = -\frac{N-p}{p} \int_{\mb R^N } |\na w|^p dx.\]
		Finally, we have 
		\begin{align*}
			&\int_{{\mathbb R^N}}\left(\int_{{\mathbb R^N}} \frac{F(y,h(w_\e))}{|y|^\beta|x-y|^\mu}dy\right)\frac{f(x,h(w_\e)) }{|x|^\beta} \psi_{r,R}(x)h'(w_\e)dx\\
			=& \frac{1}{2} \int_{\mathbb R^N}\int_{\mathbb R^N} \frac{F(y,h(w_\e))f(x,h(w_\e))h'(w_\e)\phi_{r,R}(x) x \cdot \nabla w_\e(x) }{|y|^\beta|x-y|^{\mu} |x|^\beta}dxdy\\&\qquad+ \frac{1}{2} \int_{\mathbb R^N}\int_{\mathbb R^N}\frac{F(x,h(w_\e))f(y,h(w_\e))h'(w_\e)\phi_{r,R}(y) y \cdot \nabla w_\e(y) }{|y|^\beta|x-y|^{\mu} |x|^\beta}dxdy\\
			=& -N \int_{\mathbb R^N}\int_{\mathbb R^N} \frac{F(x,h(w_\e))F(y,h(w_\e))\phi_{r,R}(x) }{|y|^\beta|x-y|^{\mu} |x|^\beta}dx dy -  \int_{\mathbb R^N}\int_{\mathbb R^N} \frac{F(x,h(w_\e))F(y,h(w_\e)) x\cdot \nabla \phi_{r,R}(x)  }{|y|^\beta|x-y|^{\mu} |x|^\beta}dx dy \\
			&\quad + \frac{\mu}{2} \int_{\mathbb R^N}\int_{\mathbb R^N} \frac{F(x,h(w_\e))F(y,h(w_\e)) }{|y|^\beta|x-y|^{\mu} |x|^\beta} \frac{(x-y)\cdot (x \phi_{r,R}(x)- y \phi_{r,R}(y))}{|x-y|^2}dx dy \\&\qquad+ \beta \int_{\mathbb R^N}\int_{\mathbb R^N}\frac{F(x,h(w_\e))F(y,h(w_\e))\phi_{r,R}(x) }{|y|^\beta|x-y|^{\mu} |x|^\beta}dxdy.
		\end{align*}
		Taking $r \ra 0$ and $R\ra \infty$, by Lebesgue's dominated convergence Theorem,
		\begin{align*}
			\int_{\mathbb R^N}\int_{\mathbb R^N} \frac{F(x,h(w_\e))F(y,h(w_\e)) }{|y|^\beta|x-y|^{\mu} |x|^\beta} \frac{(x-y)\cdot (x \phi_{r,R}(x)- y \phi_{r,R}(y))}{|x-y|^2}dx dy &\ra    \int_{\mathbb R^N}\int_{\mathbb R^N} \frac{F(x,h(w_\e))F(y,h(w_\e))}{|y|^\beta|x-y|^{\mu} |x|^\beta}dx dy;\\
			\int_{\mathbb R^N}\int_{\mathbb R^N} \frac{F(x,h(w_\e))F(y,h(w_\e)) x\cdot \nabla \phi_{r,R}(x)  }{|y|^\beta|x-y|^{\mu} |x|^\beta}dx dy &\ra 0.
		\end{align*}
		\noi Hence, combining the above estimates and taking $\e\ra 0$, we obtain
		\[\int_{{\mathbb R^N}}\left(\int_{{\mathbb R^N}} \frac{F(y,h(w_\e))}{|y|^\beta|x-y|^\mu}dy\right)\frac{f(x,h(w_\e))}{|x|^\beta}\psi_{r,R}(x) dx\\
		\ra  -\frac{2N-\mu-2\beta}{2} \int_{\mathbb R^N}\int_{\mathbb R^N} \frac{F(y,h(w))F(x,h(w))}{|y|^\beta|x-y|^{\mu} |x|^\beta}dxdy. \]
		This completes the proof pf the proposition. 
			\end{proof}
	\noindent{\bf Acknowledgements:} 
	The second author would like to thank the Science and Engineering Research Board, Department of Science and Technology,
	Government of India for the financial support under the grant\\
	MTR/2018/001267.                                                   
	
\end{document}